\documentclass[a4paper,11pt,openright,twoside]{amsart}
\usepackage{geometry}
\usepackage{etex}%
\usepackage[utf8]{inputenc}
\usepackage[english]{babel}
\usepackage[T1]{fontenc}
\usepackage{acronym}%
\usepackage{epstopdf}%
\usepackage{epigraph}%
\usepackage{indentfirst}%
\usepackage{amsmath,amssymb, amsthm, amsfonts}
\usepackage{braket}%
\usepackage{empheq}%
\usepackage{hyperref}
\usepackage{mathtools}
\usepackage{dcpic}%
\usepackage[all,cmtip,2cell]{xy} 
\UseAllTwocells
\input{diagxy}
\usepackage{wrapfig}%
\usepackage{subfig,caption}%
\usepackage{stmaryrd}%
\usepackage{textcomp}%
\usepackage{mathrsfs}%
\usepackage{lmodern}
\usepackage{tikz}
\usepackage{tikz-cd}
\usetikzlibrary{graphs}
\usetikzlibrary{quotes}
\usetikzlibrary{matrix}
\usepackage[all]{xy}

\newcommand\nbd\nobreakdash
\newcommand{\Cat}{{\mathcal{C}\mspace{-2.mu}\it{at}}}
\newcommand{\nCat}[1]{{#1}\hbox{\protect\nbd-}\kern1pt\Cat}	
\newcommand{\ulthreecell}[3][0.5]{\ar@{}[#2] \ar@3?(#1)+/dr  0.2cm/;?(#1)+/ul 0.2cm/_{#3}}
\newcommand{\parallelsum}{\mathbin{\!/\mkern-5mu/\!}}

\newcommand{\gpd}{{\mathcal{G}\mspace{-2.mu}\it{pd}}}
\newcommand{\ngpd}[1]{{#1}\hbox{\protect\nbd-}\kern1pt\gpd}
	
\newcommand{\Twocong}[2][0.5]{\ar@{}[#2] \save ?(#1)*{\cong}\restore}
\newcommand{\Twoeq}[2][0.5]{\ar@{}[#2] \save ?(#1)*{=}\restore}
\newcommand{\Ltwocell}[3][0.5]{\ar@{}[#2] \ar@{=>}?(#1)+/r 0.2cm/;?(#1)+/l 0.2cm/^{#3}}
\newcommand{\Rtwocell}[3][0.5]{\ar@{}[#2] \ar@{=>}?(#1)+/l 0.2cm/;?(#1)+/r 0.2cm/^{#3}}
\newcommand{\Utwocell}[3][0.5]{\ar@{}[#2] \ar@{=>}?(#1)+/d  0.2cm/;?(#1)+/u 0.2cm/_{#3}}
\newcommand{\Dtwocell}[3][0.5]{\ar@{}[#2] \ar@{=>}?(#1)+/u  0.2cm/;?(#1)+/d 0.2cm/^{#3}}
\newcommand{\Ultwocell}[3][0.5]{\ar@{}[#2] \ar@{=>}?(#1)+/dr  0.2cm/;?(#1)+/ul 0.2cm/^{#3}}
\newcommand{\Urtwocell}[3][0.5]{\ar@{}[#2] \ar@{=>}?(#1)+/dl  0.2cm/;?(#1)+/ur 0.2cm/^{#3}}
\newcommand{\Dltwocell}[3][0.5]{\ar@{}[#2] \ar@{=>}?(#1)+/ur  0.2cm/;?(#1)+/dl 0.2cm/^{#3}}
\newcommand{\Drtwocell}[3][0.5]{\ar@{}[#2] \ar@{=>}?(#1)+/ul  0.2cm/;?(#1)+/dr 0.2cm/^{#3}}
\newcommand{\Ulthreecell}[3][0.5]{\ar@{}[#2] \ar@3?(#1)+/dr  0.2cm/;?(#1)+/ul 0.2cm/_{#3}}

\newcommand{\D}{\mathcal{D}}

\newcommand{\G}{\mathbb{G}}

\newcommand{\C}{\mathscr{C}}

\newcommand{\Tn}{\Theta^{\leq n}_0}
\newcommand{\Mod}{\textsc{Mod}}

\newcommand{\plus}[1]{\mathop{\amalg}\limits_{#1}}
\renewcommand{\epsilon}{\varepsilon}
\renewcommand{\theta}{\vartheta}
\renewcommand{\rho}{\varrho}
\renewcommand{\phi}{\varphi}

\newcounter{ctr} \numberwithin{ctr}{section}

\theoremstyle{definition}

\theoremstyle{definition}

\newtheorem{nota}[ctr]{Notation}
\newtheorem{constr}[ctr]{Construction}
\newtheorem{defi}[ctr]{Definition}

\newtheorem{rmk}[ctr]{Remark}
\newtheorem{ex}[ctr]{Example}

\theoremstyle{plain}
\newtheorem*{Mthm}{Main theorem}
\newtheorem{thm}[ctr]{Theorem}

\newtheorem{prop}[ctr]{Proposition}
\newtheorem{lemma}[ctr]{Lemma}
\newtheorem{cor}[ctr]{Corollary}
\newtheorem*{Cor}{Corollary}

\DeclareMathOperator{\colim}{colim}

\DeclareMathOperator{\height}{ht}

\newcommand{\Ccal}{\mathcal{C}}
\newcommand{\Dcal}{\mathcal{D}}

\begin{document}
	\title{On the homotopy hypothesis in dimension 3}
	\author{Simon Henry}
	\address{Masaryk university,  Mathematics department \\ Kotl\'a\v{r}sk\'a 2 \\ 611 37 Brno \\ Czech republic}
	\email{henrys@math.muni.cz}
	\urladdr{http://www.normalesup.org/~henry/}
	\author{Edoardo Lanari}
	\address{Institute of Mathematics CAS \\ \v{Z}itn\'a 25 \\115 67   Praha 1\\ Czech Republic}
	\email{edoardo.lanari.el@gmail.com}
	\urladdr{https://sites.google.com/view/edoardo-lanari}
%	\date{\cleanlookdateon \today}
	
	\maketitle

\begin{abstract}
We show that if the canonical left semi-model structure on the category of Grothendieck $n$-groupoids exists, then it satisfies the homotopy hypothesis, i.e. the associated \((\infty,1)\)-category is equivalent to that of homotopy $n$-types, thus generalizing a result of the first named author. As a corollary of the second named author's proof of the existence of the canonical left semi-model structure for Grothendieck 3-groupoids, we obtain a proof of the homotopy hypothesis for Grothendieck 3-groupoids.
\end{abstract}

\tableofcontents

\section*{Introduction}
The generalized homotopy hypothesis, formulated in the 80's by Alexander Grothendieck, roughly states that the homotopy theories of weak \(n\)-groupoids and of homotopy \(n\)-types are equivalent. In particular, he conjectured this to be true for a particular kind of algebraic model of higher groupoids, introduced in \cite{Gr}. Later, Maltsiniotis gave a more compact definition along the lines of Grothendieck's one (see \cite{MA}), and he adapted it also to the case of (weak) \(\infty\)-categories.

Grothendieck \(\infty\)-groupoids are defined as \(\mathbf{Set}\)-valued models of a specific class of \emph{globular} theories, called coherators. These have the property of being contractible and cellular: the former provides all the appropriate operations that ought to exist for a weak \(\infty\)-groupoid, while the latter ensures that all the ``relations'' among the various operations are encoded by higher homotopies rather than on-the-nose equalities.

In \cite{EL} and \cite{EL2}, the second author introduces the notion of Grothendieck \(n\)-groupoid and addresses the problem of endowing the corresponding categories with a ``canonical'' left semi-model structure (see \cite{Ba} for the definition of these structures), i.e. the left semi-model structure whose equivalences are the maps that induce bijections on (suitably defined) homotopy groups, and the cofibration are the maps obtained by freely adding arrows with specified source and target, or identifying parallel cells of maximal dimension. 

In \cite{EL2}, in particular, it is shown that the validity of a certain technical result, called the \emph{pushout lemma}, constitutes a necessary and sufficient condition for the existence of the canonical model structure (see proposition \ref{model str} for the precise statement), and sufficient conditions for its validity are given, in terms of a (weaker version of a) path object construction. A candidate for the underlying globular set of this putative path object is given in \cite{EL}, and its full construction is achieved for \(n=3\) in \cite{EL2}, where the category of Grothendieck 3-groupoids gets endowed with a left semi-model structure.

%Note:I've changed the order because it was more correct to refer to \ref{model str} for the statement of the pushout lemma in the paragraph above the existence of the model structure than in the one about the homotopy hypothesis.

In \cite{Hen}, the first author independently developed tools that enables him to show that the homotopy hypothesis for Grothendieck $\infty$-groupoids is valid provided the same technical result, the \emph{pushout lemma}, holds true. The main result of the present paper, theorem \ref{main thm}, is an extension of this to the case of $n$-groupoids for $n < \infty$ :

\begin{Mthm}
Assuming the canonical left semi-model structure on \(\ngpd{n}\) exists, then it is equivalent to the model category for homotopy $n$-types.
\end{Mthm}

We recall in Proposition \ref{model str} that the validity of the pushout lemma (in dimension $n$) implies (actually, it is equivalent to) the existence of the canonical left semi-model structure. While the existence of the canonical left semi-model structure might seem a strong assumption, the pushout lemma appears to be a very plausible conjecture.

By putting together the main theorem with the above mentioned ones in \cite{EL2}, we obtain the following corollary.

\begin{Cor}
Grothendieck 3-groupoids satisfy the generalized homotopy hypothesis and thus provide an algebraic model for homotopy \(3\)-types.
\end{Cor} 
A similar result in the stricter case of Gray-groupoids was proven by Lack in \cite{La2}.

One of the main difficulties in adapting the idea of \cite{Hen} to the finite dimensional case is that the definition of $n$-groupoids involve equations in the highest dimension, while the argument in \cite{Hen} relies in an essential way on the absence of equations. To solve this problem, we show in Section \ref{sec:coskeletal_truncated} that $n$-groupoids can be replaced by a notion of coskeletal $(n+1)$-groupoids, which involve no equations. More precisely, we show that their homotopy theories are equivalent, even without assuming the existence of the left semi-model structure, and we show that the existence of the left semi-model structure on $n$-groupoids is equivalent to the existence of an analogue of the canonical left semi-model structure for $(n+1)$-coskeletal groupoids.

In Section \ref{sec:cylinder_cat} we revise the tools developed in \cite{Hen}, and in Section \ref{sec:mainTh} we show how they are used to prove (assuming the existence of the left semi-model structure) the equivalence between $(n+1)$-coskeletal groupoids and homotopy $n$-types, hence obtaining our main theorem.

A crucial step for this part is the identification of the universal property of the \((\infty,1)\)-category of homotopy \(n\)-types. This can be roughly formulated as follows: the \((\infty,1)\)-category of homotopy \(n\)-types is the free cocomplete \((\infty,1)\)-category on an \(n\)-co-truncated object.

\subsection*{Acknowledgements}
The first author was supported by the Operational Programme Research, Development and Education Project “Postdoc@MUNI” (No. CZ.02.2.69/0.0/0.0/16 027/0008360).

The second author gratefully acknowledges the support of Praemium Academiae of M. Markl and RVO:67985840.
\section{Background}
\label{sec:Background}

This section comprises all the preliminary definitions and constructions needed in the context of globular (weak) \(\infty\)-groupoids. For a more detailed account, see \cite{AR1} and \cite{MA}. 
\begin{defi}
	Let $\G$ be the category obtained as the quotient of the free category on the graph
%	\[\begin{tikzcd}
%	\partial D_{n+1} \ar[d,hook,shift left=10ex] \ar[r,"(sf;tf)"] & P \\  
%	D_{n+1} \ar[ur,dotted,"f_i"{swap}] 
%	\end{tikzcd} \]
	\[\begin{tikzcd}
0 \ar[r, shift left=1ex,"\sigma_0"] \ar[r, shift right=1ex,"\tau_0"{swap}] & 1 \ar[r, shift left=1ex,"\sigma_1"] \ar[r, shift right=1ex,"\tau_1"{swap}] & \ldots n \ar[r, shift left=1ex,"\sigma_n"] \ar[r, shift right=1ex,"\tau_n"{swap}] & n+1\ar[r, shift left=1ex,"\sigma_{n+1}"] \ar[r, shift right=1ex,"\tau_{n+1}"{swap}] &\ldots
	\end{tikzcd}\]
%	\[
%	\bfig
%	\morphism(0,0)|a|/@{>}@<2pt>/<300,0>[0`1;\sigma_0]
%	\morphism(0,0)|b|/@{>}@<-2pt>/<300,0>[0`1;\tau_0]
%	\morphism(300,0)|a|/@{>}@<2pt>/<300,0>[1`\ldots;\sigma_1]
%	\morphism(300,0)|b|/@{>}@<-2pt>/<300,0>[1`\ldots;\tau_1]
%	\morphism(720,0)|a|/@{>}@<2pt>/<400,0>[n`n+1;\sigma_n]
%	\morphism(720,0)|b|/@{>}@<-2pt>/<400,0>[n`n+1;\tau_n]
%	\morphism(1120,0)|a|/@{>}@<2pt>/<400,0>[n+1`\ldots;\sigma_{n+1}]
%	\morphism(1120,0)|b|/@{>}@<-2pt>/<400,0>[n+1`\ldots;\tau_{n+1}]
%	\efig
%	\]
	by the set of relations $\sigma_k \circ \sigma_{k-1}=\tau_k \circ \sigma_{k-1}$, $\sigma_k \circ \tau_{k-1}=\tau_k \circ \tau_{k-1}$ for $k\geq 1$.
	
	Given integers $j>i$, define $\sigma^j_i=\sigma_{j-1}\circ \sigma^{j-1}_i$, where $\sigma^{i+1}_i=\sigma_i$. The maps $\tau^j_i$ are defined similarly.
	
	The category of globular sets is by definition the presheaf category $[\G^{\mathrm{op}},\mathbf{Set}]$.
\end{defi}
\begin{defi}
	For  $0\leq n $, we denote with $\G_n$ the full subcategory of $\G$ generated by the set of objects $\{k\in\G \colon k\leq n\}$.
	
	The category of $n$-globular sets is by definition the presheaf category $[\G_n^{op},\mathbf{Set}]$.
\end{defi}
We now describe more complex shapes, which are a special kind of pasting of globes. These are needed to express the algebraic structure of $n$-groupoids. In what follows, we let $0\leq n \leq \infty$, where the case $n=\infty$ refers to globular sets.
\begin{defi}
%	A table of dimensions is a sequence of integers of the form 
%	\begin{equation}
%	\label{**}
%	\begin{pmatrix}
%	i_1 &&i_2 & \ldots&i_{m-1} & &i_m\\
%	& i'_1 & &\ldots&& i'_{m-1}
%	\end{pmatrix}
%	\end{equation}
%	satisfying the following inequalities: $i'_k<i_k$ and $i'_k<i_{k+1}$ for every $1\leq k\leq m-1$.
We define $\Theta_0$ as the cocompletion of $\G$ with respect to diagrams of the form: 
\[
\begin{tikzcd}
\label{glob sums}
i_1 & & i_2 & & i_3 &\ldots & i_{m-1} & & i_m\\
&i'_1 \ar[ul,"\sigma"] \ar[ur, "\tau"{swap}] & & i'_2 \ar[ul,"\sigma"] \ar[ur, "\tau"{swap}] & &\ldots  & & i'_{m-1} \ar[ul,"\sigma"] \ar[ur, "\tau"{swap}]
\end{tikzcd}\]
%\[
%\bfig
%\morphism(0,0)|l|/@{>}@<0pt>/<-250,250>[i'_1`i_1;\sigma]
%\morphism(0,0)|r|/@{>}@<0pt>/<250,250>[i'_1`i_2;\tau]
%\morphism(500,0)|l|/@{>}@<0pt>/<-250,250>[i'_2`i_2;\sigma]
%\morphism(500,0)|r|/@{>}@<0pt>/<250,250>[i'_2`i_3;\tau]
%\morphism(900,125)|r|/@{}@<0pt>/<100,0>[ `\ldots;]
%\morphism(1500,0)|l|/@{>}@<0pt>/<-250,250>[i'_{m-1}`i_{m-1};\sigma]
%\morphism(1500,0)|r|/@{>}@<0pt>/<250,250>[i'_{m-1}`i_{m};\tau]
%\efig
%\] 
$\Theta_0$ can thus be realized as a full subcategory of \([\G^{op},\mathbf{Set}]\) (see \cite{AR1} for a combinatorial description of this category).
%	Given a category $\mathcal{C}$ and a functor $F\colon \G_n \rightarrow \mathcal{C}$, a table of dimensions as above, with $i_k\leq n$ for all $1\leq k \leq m$, induces a diagram of the form
%	\[
%	\bfig
%	\morphism(0,0)|l|/@{>}@<0pt>/<-400,400>[F(i'_1)`F(i_1);F(\sigma_{i'_1}^{i_1})]
%	\morphism(0,0)|a|/@{>}@<0pt>/<400,400>[F(i'_1)`F(i_2);F(\tau_{i'_1}^{i_2})]
%	\morphism(800,0)|l|/@{>}@<0pt>/<-400,400>[F(i'_2)`F(i_2);F(\sigma_{i'_2}^{i_2})]
%	\morphism(800,0)|a|/@{>}@<0pt>/<400,400>[F(i'_2)`F(i_3);F(\tau_{i'_2}^{i_3})]
%	\morphism(1300,150)|r|/@{}@<0pt>/<100,0>[ `\ldots;]
%	\morphism(2000,0)|l|/@{>}@<0pt>/<-400,400>[F(i'_{m-1})`F(i_{m-1});F(\sigma_{i'_{m-1}}^{i_{m-1}})]
%	\morphism(2000,0)|r|/@{>}@<0pt>/<400,400>[F(i'_{m-1})`F(i_{m});F(\tau_{i'_{m-1}}^{i_{m}})]
%	\efig 
%	\]
%	The $n$-globular sum (of type $F$) associated with \eqref{**} is the colimit in $\mathcal{C}$ (if it exists) of the diagram above.
%	
	We call globular sums the objects of $\Theta_0$, and  define the height of the globular sum $A$ associated with the diagram above to be $\height(A)=\max \{i_k\}_{k \in \{1, \ldots, \ m\}}$. Given a globular sum $A$, we denote with $\iota_k^A$ the colimit inclusion $i_k\rightarrow A$, dropping subscripts when there is no risk of confusion.
\end{defi}
\begin{nota}
It is customary to denote by $D_i$ the object of $\Theta_0$ corresponding to $i\in \G$. 
%$y(i)$ and  the globular sum corresponding to the table of dimensions:
%	\[\begin{pmatrix}
%	1 &&1 & \ldots&1 & &1\\
%	& 0 & &\ldots&& 0
%	\end{pmatrix}\] by $D_1^{\otimes k}$, where the integer $1$ appears exactly $k$ times.
Also, we denote by $\Tn $ the full subcategory of $\Theta_0$ spanned by globular sums of height less or equal than $n$.
\end{nota}
%In dealing with Grothendieck $n$-groupoids, we will need a truncated version of the category $\G$, which we now introduce.
%\begin{defi}
%	We denote with $\G_n$ the full subcategory of $\G$ generated by the set of objects $\{k\in\G \colon k\leq n\}$. Analogously to the infinite dimensional case, we consider the presheaf category $[\G_{n}^{op},\mathbf{Set}]$, called the category of $n$-truncated globular sets, or simply $n$-globular sets.
%\end{defi}
%We will always assume $n>0$, to avoid the trivial case of $0$-groupoids, i.e.\ sets.
%Proposition \ref{fact syst glob set} can be extended to the case of $n$-globular sets, when (using the notation of the proposition) $m\leq n$.
It is not hard to see that there is a fully faithful embedding functor $\Tn\rightarrow [\G_{n}^{op},\mathbf{Set}]$. The category $\Tn$ plays a similar role for $n$-groupoids as $\Theta_0$ does for $\infty$-groupoids.
\begin{defi}
	An $n$-truncated globular theory is a pair $(\mathfrak{E},\mathbf{F})$, where $\mathfrak{E}$ is a category and $\mathbf{F}\colon \Tn \rightarrow \mathfrak{E}$ is a bijective on objects functor that preserves the colimits of diagrams of the form \eqref{glob sums} used to define $\Tn$.
	
	We denote by $\mathbf{GlTh_n}$ the category of $n$-globular theories and $n$-globular sums preserving functors. More precisely, a morphism $H\colon(\mathfrak{E},\mathbf{F}) \rightarrow (\mathfrak{C},\mathbf{G})$ is a functor $H\colon \mathfrak{E} \rightarrow \mathfrak{C}$ such that $\mathbf{G}=H\circ \mathbf{F}$.
\end{defi}
If there is no risk of confusion we will omit the structural map $\mathbf{F}\colon \Tn \rightarrow \mathfrak{E}$ and simply denote the globular theory $(\mathfrak{E},\mathbf{F})$ by $\mathfrak{E}$.
\begin{defi}
	\label{models defi}
	Given an $n$-globular theory $\mathfrak{E}$, we define the category of its models, denoted $\mathbf{Mod}(\mathfrak{E})$, to be the category of functors $G\colon \mathfrak{E}^{op} \rightarrow \mathbf{Set}$ that preserve \(n\)-globular sums. Clearly, the Yoneda embedding $y\colon \mathfrak{E} \rightarrow [\mathfrak{E}^{op},\mathbf{Set}]$ factors through $\mathbf{Mod}(\mathfrak{E})$, and it will still be denoted by $y\colon \mathfrak{E} \rightarrow \mathbf{Mod}(\mathfrak{E})$. Also, notice that $\mathbf{Mod}(\Tn)\cong [\G_{n}^{op},\mathbf{Set}]$. Again, we denote the image of $i$ under $y$ by $D_i$, and we let \(\partial D_k\) be the free model on a pair of parallel \((k-1)\)-cells. Equivalently, we have \(\mathbf{Mod}(\mathfrak{E})(\partial D_k,X)\cong \{(a,b)\in X_{k-1}\times X_{k-1}\colon \epsilon(a)=\epsilon(b), \epsilon=s,t \}\). We denote the notion of parallelism between cells \(a,b\) by \(a \parallelsum b\).
\end{defi}
We now record the universal property of the category of models of an $n$-globular theory.
\begin{prop}
	\label{UP of models}
	Given an $n$-globular theory $\mathfrak{E}$, its category of models $\mathbf{Mod}(\mathfrak{E})$ enjoys a universal property: given any cocomplete category $\D$, a cocontinuous functor $F\colon \mathbf{Mod}(\mathfrak{E}) \rightarrow \D$ is determined up to a unique isomorphism by an $n$-globular sums preserving functor $\overline{F}\colon \mathfrak{E} \rightarrow \D$, corresponding to its restriction along the Yoneda embedding. Conversely, any such functor $\overline{F}\colon \mathfrak{E} \rightarrow \D$ extends in an essentially unique way to a cocontinuous one on $\mathbf{Mod}(\mathfrak{E})$.
\end{prop}
%	\begin{proof}
%		The presheaf category $[\mathfrak{E}^{op},\mathbf{Set}]$ is the free cocompletion of $\mathfrak{E}$, therefore we get a natural equivalence induced by the Yoneda embedding of the form:
%		\[[\mathfrak{E},\D]\cong [[\mathfrak{E}^{op},\mathbf{Set}],\D]_c\] where $[\cdot,\cdot]_c$ denotes the class of cocontinuous functors. It is easy to check that this restricts to an equivalence of the form:
%		\[[\mathfrak{E},\D]_{gl}\cong [\mathbf{Mod}(\mathfrak{E}),\D]_c\] where $[\cdot,\cdot]_{gl}$ denotes the set of globular sum-preserving maps.
%	\end{proof}
Grothendieck groupoids are presented as models of a certain class of globular theories, namely the cellular and contractible ones.
\begin{defi}
	\label{contr glob th}
	Given $k\leq n$, two maps $f,g\colon D_k\rightarrow A$ in an $n$-globular theory \(\mathfrak{E}\) are said to be parallel if either $k=0$ or $f\circ \epsilon= g\circ \epsilon$ for $\epsilon=\sigma,\tau$.
	A pair of parallel maps $(f,g)$ is said to be admissible if $\height(A) \leq k+1$.
	A globular theory $(\mathfrak{C},F)$ is called contractible if for every admissible pair of maps $f,g\colon D_k\rightarrow A$ either $k=n$ and $f=g$, or $k<n$ and there exists an extension $h\colon D_{k+1}\rightarrow A$ rendering the following diagram serially commutative
	\[
	\begin{tikzcd}
D_k \ar[r,shift left=1ex,"f"] \ar[r,shift right=1ex,"g"{swap}] \ar[d, shift left=1ex, "\tau_k"] \ar[d, shift right=1ex, "\sigma_k"{swap}]& A\\
D_{k+1} \ar[ur,"h"{swap}]
	\end{tikzcd}
	\]
%	\[
%	\bfig 
%	\morphism(0,0)|a|/@{>}@<3pt>/<500,0>[D_k`A;f]
%	\morphism(0,0)|b|/@{>}@<-3pt>/<500,0>[D_k`A;g]
%	
%	\morphism(0,0)|r|/@{>}@<3pt>/<0,-400>[D_k`D_{k+1};\tau_k]
%	\morphism(0,0)|l|/@{>}@<-3pt>/<0,-400>[D_k`D_{k+1};\sigma_k]
%	
%	\morphism(0,-400)|r|/@{>}@<0pt>/<500,400>[D_{k+1}`A;h]
%	\efig 
%	\]
\end{defi}
%Contractibility ensures the existence of all the operations that ought to be part of the structure of an $n$-groupoid. However, it does not guarantee weakness of the models, and indeed there exists a contractible globular theory (which we denote by $\tilde{\Theta}^{\leq n}$) whose models are strict $n$-groupoids.
%
%To remedy this, we need the concept of cellularity, or freeness, to restrict the class of globular theories we consider. 
The following notion is based on a slight variation of a construction explained in paragraph 4.1.3 of \cite{AR1},  which we record in the following proposition.
\begin{prop}
	\label{univ prop of glob th}
	Given an $n$-globular theory $\mathfrak{E}$ and a set $X$ of admissible pairs in it, there exists another $n$-globular theory $\mathfrak{E}[X]$ equipped with a morphism $\phi\colon \mathfrak{E} \rightarrow \mathfrak{E}[X]$ in $\mathbf{GlTh_n}$ with the following universal property: given an $n$-globular theory $\mathfrak{C}$, a morphism $H\colon\mathfrak{E}[X] \rightarrow \mathfrak{C}$ is determined up to a unique isomorphism by $F\overset{def}{=}H\circ \phi$, together with a choice of an extension to $D_{k+1}$ of the image under $F$ of each admissible pair $f,g\colon D_k \rightarrow A$ in $X$ with $k<n$, or the requirement that $F(f)=F(g)$ if $k=n$.
\end{prop}
In words, $\mathfrak{E}[X]$ is obtained from $\mathfrak{E}$ by universally adding a lift for each pair in $X$ of non-maximal dimension and by equalizing parallel $n$-dimensional operations in $X$.
\begin{defi}
	\label{cell glob th}
	An $n$-globular theory $\mathfrak{E}$ is said to be cellular if there exists a functor $\mathfrak{E}_{\bullet} \colon \omega \rightarrow \mathbf{GlTh_n}$, where $\omega$ is the first infinite ordinal, such that:
	\begin{enumerate}
		\item $\mathfrak{E}_0 \cong \Tn$;
		\item for every $m \geq 0$, there exists a family $X_m$ of admissible pairs of arrows in $\mathfrak{E}_m$ (as in Definition \ref{contr glob th}) such that $\mathfrak{E}_{m+1}\cong \mathfrak{E}_m[X_m]$;
		\item $\colim_{m \in \omega}\mathfrak{E}_{m}\cong \mathfrak{E}$.
	\end{enumerate}
	Equivalently, one can consider arbitrary ordinals $\gamma$ and assume $X_{\alpha}$ to be a singleton for each $\alpha<\gamma$.
\end{defi}
As anticipated earlier, we now define the class of $n$-globular theories which are appropriate to develop a theory of $n$-groupoids.
\begin{defi}
	\label{n groupoid defi}
	An $n$-truncated (groupoidal) coherator, or, briefly, an $n$-coherator, is a cellular and contractible $n$-globular theory. Given an $n$-coherator $\mathfrak{G}$, the category of $n$-groupoids of type $\mathfrak{G}$ is the category $\mathbf{Mod}(\mathfrak{G})$ of models of $\mathfrak{G}$, which we will sometime simply denote by $\ngpd{n}$ with no mention of the coherator itself, when this does not affect the exposition. 
%	In what follows, $\mathfrak{G}$ will always denote a coherator for $n$-groupoids, with $0\leq n \leq \infty$, and sometimes we will denote the category of its models by $\ngpd{n}$, with no reference to $\mathfrak{G}$.
\end{defi}
The restriction of an $n$-groupoid $X\colon\mathfrak{G}^{op} \rightarrow \mathbf{Set}$ to ${\Tn}^{op}$ gives an object of $\mathbf{Mod}(\Tn)\simeq [\G_n^{op},\mathbf{Set}]$, which we call the underlying $n$-globular set of $X$. The set $X_i$ represents the set of $i$-cells of $X$ for each $i\leq n$.

The definition of the homotopy group $\pi_n$ for $\infty$-groupoids can be found in \cite{MA} or \cite{AR2} and it also applies to $n$-groupoids. A morphism of $n$-groupoids is said to be a weak equivalence if it induces a bijection on these homotopy groups.

The following result gives a conditional proof of the existence of a left semi-model structure (simply called semi-model structure in \cite{EL2}) on the category of Grothendieck \(n\)-groupoids for \(0\leq n \leq \infty\). It follows from the results in Section 3 of \cite{EL2}.

\begin{prop}
	\label{model str} Fix an $n$-coherator $\mathfrak{G}$. If for all $n$-groupoid $X$ of type $\mathfrak{G}$ and any morphism $f:D_k \rightarrow X$ the natural map 
 \[ X\rightarrow X\plus{D_n}D_{n+1}\]
 induced by \(f\) and \(\sigma \colon D_n \rightarrow D_{n+1}\) is a weak equivalence, then there exists a cofibrantly generated left semi-model structure on the category of \(n\)-groupoids of type $\mathfrak{G}$  whose generating cofibrations (resp. trivial cofibrations) are the boundary inclusions \(\{\partial D_n \rightarrow D_n\}_{n \geq 0}\) (resp. the source maps \(\{D_n\rightarrow D_{n+1}\}_{n \geq 0}\)).

Its weak equivalences are the weak equivalences of $n$-groupoids as defined above. We call this structure the \emph{canonical} one.
\end{prop} 

In fact, it is sufficient to restrict to $n$-groupoids $X$ which are cofibrant with respect to said cofibrations, which is a necessary and sufficient condition for the existence of this left semi-model structure.

\section{$n$-coskeletal and $n$-truncated models}
\label{sec:coskeletal_truncated}

In this section we introduce the notion of $n$-truncated and $n$-coskeletal model of a globular theory. We prove that the category of $n$-groupoids is equivalent to that of $n$-truncated $\infty$-groupoids. We then consider an $n$-globular theory $\C^{\leq n}$ suitably obtained from the $n$-th stage of a tower for a coherator $\C$ for $\infty$-groupoids, and we show how to recover the category of its models by looking at the $n$-coskeletal $\C$-models. Finally, we show that the homotopy theories modeled by $n$-groupoids and $(n+1)$-coskeletal $\infty$-groupoids coincide.
\begin{defi}
Define a functor $t_n\colon \G \rightarrow \G_n$ whose action on objects is given by:
\[t_n D_k \overset{def}{=}\begin{cases}
D_k & k \leq n\\
D_n & k >n
\end{cases}
\]  with the unique possible extension to morphisms. This induces an adjunction of the form:
\begin{equation}
\begin{tikzcd}
\label{truncation adj}
{[}\G^{op},\mathbf{Set}{]} \ar[r,bend left,"{t_n}_!"]\ar[r, phantom, "\perp" description]& {[}\G_n^{op},\mathbf{Set}{]} \ar[l,bend left, "t_n^*"]
\end{tikzcd}
\end{equation}
%\begin{equation}
%
% \xymatrixcolsep{1pc}
%\vcenter{\hbox{\xymatrix{
%			**[l][\G^{op},\mathbf{Set}] \xtwocell[r]{}_{t_n^*}^{{t_n}_!}{'\perp}& **[r] [\G_n^{op},\mathbf{Set}]}}}
%\end{equation}
where $t_n^*$ is given by precomposition and its left adjoint \({t_n}_!\) is the left Kan extension $Lan_{\mathrm{y}}(\mathrm{y}\circ t_n)$ displayed below:
\[
\begin{tikzcd}
\G \ar[r,"t_n"] \ar[d,"y" {swap}] &\G_n \ar[r,"y"] & {[}\G_n^{op},\mathbf{Set}{]}\\
{[}\G^{op},\mathbf{Set}{]} \ar[urr,"{t_{n}}_!"{swap}]
\end{tikzcd}
\]
\end{defi}
It is easy to prove that one has:
\[(t_n^*X)_k =\begin{cases}
X_k & k \leq n\\
X_n & k >n
\end{cases}\] and for $\epsilon=\sigma,\tau$ we have $ (t_n^*X)(\epsilon_k)=X(\epsilon_k)$ if $k\leq n$ and the identity otherwise. Similarly, we get 
\[({t_n}_!Y)_k =\begin{cases}
Y_k & k < n\\
Y_n/\sim & k >n
\end{cases}\] where $\sim$ is the equivalence relation generated by posing that given two $n$-cells $x,y$ in $Y$, we have that $x\sim y$ if there is an $(n+1)$-cell $H\colon x \rightarrow y$.
\begin{defi}
A globular set $X$ is called \emph{$n$-truncated} if the unit of the adjunction ${t_n}_!\dashv  t_n^*$ is an isomorphism at $X$, i.e. $\eta_X\colon X \overset{\cong}{\longrightarrow} t_n^*\circ {t_n}_!X$.
\end{defi}
The following result is straightforward.
\begin{prop}
	\label{base case}
The adjunction in \eqref{truncation adj} restricts to an equivalence of categories between the category of $n$-truncated globular sets and that of $n$-globular sets.
\end{prop}
We can give the following useful characterization of $n$-truncated globular sets as follows.
\begin{lemma}
A globular set $X$ is $n$-truncated if and only if for each pair of $k$-cells $x,y$ in $X$ with $k\geq n$ one has:
\[X(x,y)=\begin{cases}
\mathbf{1} & x=y\\
\emptyset & x \neq y
\end{cases}\] where $X(x,y)$ is the globular set whose $m$-cells are $(m+k+1)$-cells in $X$ whose $k$-dimensional boundary is given by $(x,y)$, and $\mathbf{1}$ denotes the terminal globular set.
\end{lemma}
\begin{defi}
Given a globular theory $\C$, we say that a $\C$-model $X$ is $n$-truncated if its underlying globular set is such.  We denote the full subcategory of $\C$-models spanned by $n$-truncated objects by $\mathbf{Mod}(\C)_{n-tr}$.
\end{defi}
Let $\C$ be a coherator for $\infty$-groupoids. We want to define an $n$-truncated coherator out of $\C$, denoted with $\overline{\C}^{\leq n}$ in such a way that we have an equivalence of categories of the form:
\[\mathbf{Mod}(\C)_{n-tr}\simeq \mathbf{Mod}(\overline{\C}^{\leq n})\] 
We define it via a tower of length $n+1$, in the following manner: consider the defining tower of $\C$, consisting of a functor $\C_{\bullet}\colon \omega\rightarrow \mathbf{GlTh}$ into the category of globular theories and globular maps. Without loss of generality we can assume $\C_{k+1}=\C_k[A_k]$ where each pair $(h_1,h_2)\in A_k$ is of the form $\partial D_{k+1}\rightarrow B$ for some globular sum $B$. Informally, this means that we only add $(k+1)$-dimensional operation when going from $C_k$ to $C_{k+1}$. This assumption lets us construct a tower of $n$-globular theories of length n+1 $\C^{\leq n}_{\bullet}\colon \mathbf{n+1}\rightarrow \mathbf{GlTh_n}$ by setting $\C^{\leq n}_{k+1}\overset{def}{=}\C^{\leq n}_k[A'_k]$, where $A'_k$ corresponds to $A_k$ under the identification of $\C^{\leq n}_k$ with the full subcategory of $\C_k$ on globular sums of dimension less than or equal to $n$. This holds since both categories have the same universal property. Finally, set $\overline{\C^{\leq n}}=\C^{\leq n}_{n+1}\overset{def}{=}\C^{\leq n}_n[A'_n]$, i.e. the universal globular theory obtained by identifying pairs in $A'_n$.
%\[\overline{\C}^{\leq n}_k\overset{def}{=}\C^{\leq n}_k\] for every $k\leq n$ and \[\overline{\C}^{\leq n}=\overline{\C}^{\leq n}_{n+1}\overset{def}{=}\overline{\C}^{\leq n}_n/[A_n]\] 
\begin{prop}
	There is an equivalence of categories of the form:
	\begin{equation}
	\label{truncation eq}
	\mathbf{Mod}(\C)_{n-tr}\simeq  \mathbf{Mod}(\overline{\C}^{\leq n})\simeq \ngpd{n}
	\end{equation} 
	Moreover, this adjunction is compatible with the one depicted in \eqref{truncation adj} at the level of underlying globular sets.
\end{prop}
\begin{proof}
	We prove by induction on $k$ that there is an equivalence of categories of the form $\mathbf{Mod}(\C_k)_{n-tr}\simeq \mathbf{Mod}(\overline{\C}^{\leq n}_k)$, where we extend the tower of $\overline{\C}^{\leq n}$ with identities. When $k=0$ this is precisely the content of Proposition \ref{base case}. Now let's assume the statement holds for $k$, and let's prove it for $k+1$. Let $0<k<n$, by definition, an $n$-truncated $\C_{k+1}$-model $X$ is precisely an  $n$-truncated $\C_{k}$-model together with structural maps $X(\rho)\colon X(B)\rightarrow X(k+1)$ with $s(X(\rho))=X(h_1)$ and $t(X(\rho))=X(h_2)$ for every pair $(h_1,h_2)\colon \partial D_{k+1} \rightarrow B$ in $A_k$. By inductive assumption, this corresponds exactly to a $\C^{\leq n}_k$-model with interpretation of all maps added as fillers of pairs in $A'_k$ which is by definition a $\C^{\leq n}_{k+1}$-model.
	
	Now let $k=n+1$. By definition, an $n$-truncated $\C_{n+1}$-model $X$  is precisely an  $n$-truncated $\C_{n}$-model together with structural maps $X(\rho)\colon X(B)\rightarrow X(n+1)$ with $s(X(\rho))=X(h_1)$ and $t(X(\rho))=X(h_2)$ for every pair $(h_1,h_2)\colon \partial D_{n+1} \rightarrow B$ in $A_n$. But, since $X$ is $n$-truncated, we have that $X(\sigma_{n+1})=X(\tau_{n+1})=\mathbf{Id}_{X_n}$, so that this corresponds, again by inductive assumption, to a $\C^{\leq n}_n$-model which equalizes each pair $(h_1,h_2)$ in $A'_n$. This is clearly equivalent to being a $\overline{\C}^{\leq n}_{n+1}$-model.
	
	Finally, let $k\geq n+1$. To conclude the proof it is enough to show that the category of $n$-truncated $\C_{k+1}$-models is equivalent to that of $n$-truncated $\C_{k}$-models. Arguing as before, we see that if $(h_1,h_2)$ is a pair in $A_k$ and $X$ is $n$-truncated, then in particular $h_1$ and $h_2$ are parallel, which gives us $X(h_1)=s(X(h_1))=X(s(h_1))=X(s(h_2))=s(X(h_2))=X(h_2)$ for dimensionality reasons (since $X$ is $n$-truncated), so that the interpretation of $X(\rho)$ exists and is uniquely given by $X(h_1)=X(h_2)$. This proves that $\mathbf{Mod}(\C_k)_{n-tr} \simeq \mathbf{Mod}_n(\C_{k+1})_{n-tr}$ which concludes the proof.
\end{proof}
Consider now the inclusion $\iota_n \colon \G_n \rightarrow \G$, which induces an adjunction of the form:
\begin{equation}
\label{cosk adj}
\begin{tikzcd}
{[}\G^{op},\mathbf{Set}{]} \ar[r,bend left,"{\iota_n}^*"]\ar[r, phantom, "\perp" description]& {[}\G_n^{op},\mathbf{Set}{]} \ar[l,bend left, "{\iota_n}_*"]
\end{tikzcd}
\end{equation}
%\xymatrixcolsep{1pc}
%\vcenter{\hbox{\xymatrix{
%			**[l]\mathbf{[\G^{op},\mathbf{Set}]} \xtwocell[r]{}_{{\iota_n}_*}^{\iota^*_n}{'\perp}& **[r] [\G_n^{op},\mathbf{Set}]}}}
 as before.
Explicitely, we have:
\[
(\iota^*_n X)_k=X_k \ \text{for every} \ k\leq n \] and
\[({\iota_n}_*X)_k=\begin{cases}
X_k & k\leq n\\
\{(x,y)\in X_{k-1}\times X_{k-1} \colon x \parallelsum y\}& k >n 
\end{cases}
\] 
Informally, in ${\iota_n}_*X$ there is exactly one $(k+1)$-cell between any pair of parallel $k$-cells whenever $k\geq n$.
\begin{defi}
A globular set $X$ is \emph{$n$-coskeletal} if the unit of the adjunction $\iota^*_n\dashv {\iota_n }_*$ is an isomorphism at $X$, i.e. $\eta_X\colon X \overset{\cong}{\longrightarrow} {\iota_n }_* \iota^*_n X$.
\end{defi}
Just like we did above, we record the following straightforward result.
\begin{prop}
	\label{cosk eq}
The adjunction in \eqref{cosk adj} restricts to an equivalence of categories between the category of $n$-coskeletal globular sets and that of $n$-globular sets.
\end{prop}
\begin{defi}
	Given a globular theory $\C$, we say that a $\C$-model $X$ is $n$-coskeletal if its underlying globular set is such. We denote the full subcategory of $\C$-models spanned by $n$-coskeletal objects by $\mathbf{Mod}(\C)_{cosk_n}$.
\end{defi}
We can give the following useful characterization for the classes of $n$-coskeletal globular sets.
\begin{prop}
	\label{cosk as orthogonal class}
A globular set $X$ is $n$-coskeletal if and only if for each pair of parallel $k$-cells $x\parallelsum y$ in $X$ with $k\geq n$ one has:
\[X(x,y)=\mathbf{1}\] where $X(x,y)$ is the globular set whose $m$-cells are $(m+k+1)$-cells in $X$ whose $k$-dimensional boundary is given by $(x,y)$, and $\mathbf{1}$ denotes the terminal globular set.
		
Equivalently, the category of $n$-coskeletal globular sets consists of the full subcategory of the category of globular sets spanned by the objects which are orthogonal to the set of maps $\{S^{k-1}\rightarrow D_k\}_{k\geq n+1}$.
% \[\{(\mathbf{Id}_{D_k},\mathbf{Id}_{D_k})\colon \partial D_{k+1}\rightarrow D_k\}_{k >n}\cup \{ D_{n+1}\rightarrow D_n\plus{\partial D_{n+1}}D_{n+1}\}\cup \{\sigma_k\colon D_k \rightarrow D_{k+1}\}_{k\geq n} \] 
%where the second map is defined by the following cocartesian square:
%\[
%\bfig
%\morphism(0,0)|a|/@{>}@<0pt>/<600,0>[\partial D_{n+1}`D_{n+1};j]
%\morphism(0,0)|a|/@{>}@<0pt>/<0,-400>[\partial D_{n+1}`D_{n};(\mathbf{Id}_{D_n},\mathbf{Id}_{D_n})]
%\morphism(600,0)|a|/@{>}@<0pt>/<0,-400>[D_{n+1}`D_n\plus{\partial D_{n+1}}D_{n+1};]
%\morphism(0,-400)|a|/@{>}@<0pt>/<600,0>[D_n`D_n\plus{\partial D_{n+1}}D_{n+1};]
%\efig 
%\]
%Similarly, the category of $n$-coskeletal globular sets is the full subcategory of the category of globular sets spanned by the objects which are orthogonal to the set of maps 
%\[\{(\mathbf{Id}_{D_k},\mathbf{Id}_{D_k})\colon \partial D_{k+1}\rightarrow D_{k+1}\}_{k\geq n}\]
\end{prop}
% Let $\C$ be a coherator for $\infty$-groupoids, and consider its defining tower, consisting of a functor $\C_{\bullet}\colon \omega\rightarrow \mathbf{GlTh}$ into the category of globular theories and globular maps. Without loss of generality we can assume $\C_{k+1}=\C_k[A_k]$ where each pair $(h_1,h_2)\in A_k$ is of the form $\partial D_{k+1}\rightarrow B$ for some globular sum $B$. Informally, this means that we only add $(k+1)$-dimensional operation when going from $C_k$ to $C_{k+1}$. This assumption lets us construct a tower of $n$-globular theories of length n $\C^{\leq n}_{\bullet}\colon \mathbf{n}\rightarrow \mathbf{GlTh_n}$ by setting $\C^{\leq n}_{k+1}\overset{def}{=}\C^{\leq n}_k[A'_k]$, where $A'_k$ corresponds to $A_k$ under the identification of $\C^{\leq n}_k$ with the full subcategory of $\C_k$ on globular sums of dimension less than or equal to $n$. This holds since both categories have the same universal property. Finally, set $\C^{\leq n}\overset{def}{=}\C^{\leq n}_n$.

\begin{prop}
\label{n-cosk models}
There is an equivalence of categories between $n$-coskeletal $\C$-models and $\C^{\leq n}$-models, i.e.
\[\mathbf{Mod}(\C)_{cosk_n}\simeq \mathbf{Mod}(\C^{\leq n})\] 
Moreover, this adjunction is compatible with the one depicted in \eqref{cosk adj} at the level of underlying globular sets.
\end{prop}
\begin{proof}
We prove the statement by induction on the stage $k$ of their respective defining towers, i.e. we prove that $\mathbf{Mod}(\C_k)_{cosk_n}\simeq \mathbf{Mod}(\C^{\leq n}_k)$ for every $k$ (where we extend the tower of $n$-globular theories with identities). The case $k=0$ has been proven in Proposition \ref{cosk eq}. 

Now suppose it holds for $0<k<n$ and let's prove it for $k+1$. Observe that, by definition, a $\C_{k+1}$ model $X$ precisely corresponds to a $\C_k$ model together with an interpretation $X(\rho)\colon X(A)\rightarrow X(k+1)$ of every map $\rho$ added as a filler of $(h_1,h_2)\colon \partial D_{k+1} \rightarrow A$ with $(h_1,h_2) \in A_k$, satisfying the property that $X(\sigma)\circ X(\rho)=X(h_1)$ and $X(\tau)\circ X(\rho)=X(h_2)$. By inductive hypothesis, if $X$ is $n$-coskeletal, then it corresponds to a $\C^{\leq n}_k$-model with interpretation for every filler of pairs in $A'_k$, which is precisely a  $\C^{\leq n}_{k+1}$-model.

To conclude the proof, it is enough to show that, for every $k\geq n$, the category of $n$-coskeletal $\C_k$-models and that of $n$-coskeletal $\C_{k+1}$-models are equivalent. Given any $(h_1,h_2)\colon \partial D_{k+1} \rightarrow A$ in $A_k$, we have to define a map $X(\rho)\colon X(A)\rightarrow X(k+1)$ satisfying the property that $X(\sigma)\circ X(\rho)=X(h_1)$ and $X(\tau) \circ X(\rho)=X(h_2)$. This is accomplished by observing that, by Proposition \ref{cosk as orthogonal class}, $X$ admits unique fillers of $k$-spheres.
\end{proof}
\begin{rmk}
It is possible to build a coherator $\C$ for $\infty$-groupoids given a coherator $\C'$ for $n$-groupoids, in such a way that the $n$-truncation of $\C$ described above coincides with $\C'$. The procedure, which we just sketch, goes as follows: the first $n$-steps of the tower for $\C$ corresponds to those for $\C'$, i.e. we choose the same pairs to which we add fillers. Then we define $\C_k$ for $k>n$ by adding fillers \emph{for all} possible pairs of parallel maps in $\C_{k-1}$. We leave the proof that this indeed produces a coherator for \(\infty\)-groupoids and that its \(n\)-truncation coincides with $\C'$ to the interested reader.
\end{rmk}
We now turn to examine how $n$-truncated and $n$-coskeletal objects are related to one another.
	\begin{lemma}
		Every $n$-truncated globular set is $(n+1)$-coskeletal.
	\end{lemma}
	\begin{proof}
		Let $X$ be an $n$-truncated globular set. Thanks to the previous proposition, we have to show that $X(x,y)=\mathbf{1}$ for every pair of parallel $(n+1)$-cells $x,y$, which is true since $x\parallelsum y$ implies $x=y$ by $n$-truncatedness of $X$.
	\end{proof}
In particular, the adjunction in \eqref{truncation adj} restricts to one of the form:
\begin{equation}
\label{tr-cosk adj}
\begin{tikzcd}
{[}\G^{op},\mathbf{Set}{]}_{cosk_{(n+1)}} \ar[r,bend left,"{t_n}_!"] \ar[r, phantom, "\perp" description, near start]& {[}\G^{op},\mathbf{Set}{]} _{n-tr}\ar[l,bend left, "t_n^*"]
\end{tikzcd}
\end{equation}
%\xymatrixcolsep{1pc}
%\vcenter{\hbox{\xymatrix{
%			**[l][\G^{op},\mathbf{Set}]_{(n+1)-cosk} \xtwocell[r]{}_{t_n^*}^{{t_n}_!}{'\perp}& **[r] [\G^{op},\mathbf{Set}]_{n-tr}}}}

where the category on the left denotes the full subcategory of globular sets spanned by $(n+1)$-coskeletal globular sets, and that on the right denotes the one spanned by $n$-truncated globular sets. It is actually possible to extend further this adjunction, as we record here below.
\begin{prop}
There is an adjunction of the form:
\begin{equation}
\label{tr-cosk adj on models}
\begin{tikzcd}
\mathbf{Mod}(\C)_{cosk_{(n+1)}} \ar[r,bend left,"{t_n}"]\ar[r, phantom, "\perp" description, near start]   &\mathbf{Mod}(\C)_{n-tr} \ar[l,bend left, "i_n"]
\end{tikzcd}
\end{equation}
%\xymatrixcolsep{1pc}
%\vcenter{\hbox{\xymatrix{
%			**[l]\mathbf{Mod}(\C)_{(n+1)-cosk} \xtwocell[r]{}_{i_n}^{t_n}{'\perp}& **[r] \mathbf{Mod}(\C)_{n-tr}}}}

which is compatible with the one depicted in \eqref{tr-cosk adj} at the level of underlying globular sets.
\end{prop}
\begin{proof}
By definition, given an \((n+1)\)-coskeletal \(\C\)-model \(X\), we have:
\[(t_n X)_k=\begin{cases}
X_k & k <n\\
X_n/\sim & k \geq n
\end{cases}\] where source and target maps in dimension above $n$ are given by identities. We know that $X$ is equivalently a $\C_{n+1}$-model thanks to Proposition \ref{n-cosk models}, since it is $(n+1)$-coskeletal. 
There is a unique structure of $\C_{n+1}$-model on $t_n X$ such that the natural map \(X\rightarrow t_n X\) is a morphism of \(\C_{n+1}\)-models (here, we are implicitely identifying \((n+1)\)-coskeletal \(\C\)-models and \(\C_{n+1}\)-models).
%There is a unique structure of $\overline{\C}^{\leq n}$-model on $t_n X$ making the following diagram commute:
%\[
%\begin{tikzcd}
%\C_n^{\mathrm{op}} \ar[r] \ar[d] &\C_{n+1}^{\mathrm{op}} \ar[d,"X"]\\
%{\C^{\leq n}_{n+1}}^{\mathrm{op}}\ar[r,"t_n X"]&\mathbf{Set}
%\end{tikzcd}
%\]
%\[
%\bfig 
%\morphism(0,0)|a|/@{>}@<0pt>/<600,0>[\C_n^{\mathrm{op}}`\C_{n+1}^{\mathrm{op}};]
%\morphism(0,0)|a|/@{>}@<0pt>/<0,-400>[\C_n^{\mathrm{op}}`{\C^{\leq n}_{n+1}}^{\mathrm{op}};]
%\morphism(0,-400)|a|/@{>}@<0pt>/<600,0>[{\C^{\leq n}_{n+1}}^{\mathrm{op}}`\mathbf{Set};t_n X]
%\morphism(600,0)|r|/@{>}@<0pt>/<0,-400>[\C_{n+1}^{\mathrm{op}}`\mathbf{Set};X]
%\efig 
%\]
%where the top horizontal arrow is that in the tower for $\C$ and the left-hand side vertical map is the quotient map that identifies pairs in $A_n$. 
Indeed, a \(\C_{n+1}\)-model structure on \(t_n X\) is equivalent to that of a \(\overline{\C}^{\leq n}\)-model, since \(t_n X\) is \(n\)-truncated by definition. This can be obtained as follows: given an operation \(\rho\colon D_n \rightarrow A\) added in going from \(\C_{n-1}\) to \(\C_n\) (i.e. \(\rho\) belongs to \(A'_{n-1}\) in the previously established notation) we set \(t_n X (\rho)(x)\overset{def}{=}[X(\rho)(x)]\) for every map \(x\colon A \rightarrow t_n X\). The fact that this is well defined can be checked as follows: given any pair $(h,h')\colon \partial D_{n+1} \rightarrow A$ in $A_n$ with added filler $\rho$ in $\C_{n+1}$, we have $s(X(\rho)(x))=X(h)(x)$ and $t(X(\rho)(x))=X(h')(x)$, and therefore $X(h)(x)$ and $X(h')(x)$ are identified after applying the truncation functor. 

Let's prove that given any $f\colon X \rightarrow i_n Y$ there exists a unique map of $n$-truncated models $ \bar{f}\colon t_n X \rightarrow Y$ such that $i_n(\bar{f}) \circ \eta_x =f$. We know there is a unique map $\bar{f}$ between the underlying globular sets thanks to \eqref{tr-cosk adj}, so we only have to prove it preserves operations. It is enough to show it preserves the generating operations, i.e. the ones added as fillers in the tower associated with $\C$. In fact, it is enough to prove it for operations $\rho \colon D_k \rightarrow A$ with $\height(A)\leq k$. Consider the square
\[
\begin{tikzcd}
t_nX(A)\ar[r,"t_n X(\rho)"] \ar[d,"\bar{f}"{swap}] & t_nX(k) \ar[d,"\bar{f}"]\\
Y(A) \ar[r,"Y(\rho)"]&Y(k)
\end{tikzcd}
\]
%\[
%\bfig 
%\morphism(0,0)|a|/@{>}@<0pt>/<700,0>[t_nX(A)`t_nX(k);t_n X(\rho)]
%\morphism(0,0)|a|/@{>}@<0pt>/<0,-400>[t_nX(A)`Y(A);\bar{f}]
%\morphism(0,-400)|a|/@{>}@<0pt>/<700,0>[Y(A)`Y(k);Y(\rho)]
%\morphism(700,0)|r|/@{>}@<0pt>/<0,-400>[t_nX(k)`Y(k);\bar{f}]
%\efig 
%\]
It certainly commutes if $k<n$, since in that range of dimensions $t_n X$ coincides with $X$, and we already know that $f$ is a map of models. If $k=n$ then we can use the fact that the map of sets $(\eta_X)_A\colon X(A)\rightarrow t_n(X)(A)$ which sends an element $x\in X(A)$ to the equivalence class it represents, denoted by $[x]$, is surjective, and that  $t_n X(\rho)([x])=[X(\rho)(x)]$. It follows that the diagram commutes since $f$ is a map of models.

Finally, if $k>n$, then we have $(\bar{f}\circ t_nX(\rho))\circ \epsilon=\bar{f}\circ t_n X(\rho\circ \epsilon)=Y( \rho\circ \epsilon)\circ \bar{f}= (Y(\rho)\circ \bar{f})\circ \epsilon$ where $\epsilon=\sigma,\tau$. This implies that the square commutes since $Y$ is $n$-truncated, and concludes the proof.
\end{proof}
We want to use the functor:
\[\mathbf{Mod}(\C)_{cosk_{(n+1)}}\overset{t_n}{\longrightarrow} \mathbf{Mod}(\C)_{n-tr}\simeq \ngpd{n}\]
to define two classes of maps $\mathcal{W},\mathcal{W'}$ in $\mathbf{Mod}(\C)_{cosk_{(n+1)}}$. We define $\mathcal{W}$ to be the class of weak equivalences of $\infty$-groupoids between $(n+1)$-coskeletal $\C$-models, and $\mathcal{W'}$ is defined to be $t_n^{-1}(\mathcal{W}_n)$, where $\mathcal{W}_n$ is the class of weak equivalences of $n$-groupoids.
\begin{prop}
	\label{pi_n compatibility}
Given an $(n+1)$-coskeletal $\C$-model $X$ we have $\pi_k(X,a,b)\cong \pi_k(t_n X,a,b)$ for every $k\leq n$ and every pair of $(k-1)$-cells $a,b$ in $X$ (or, equivalently, in $t_n X$).

Similarly, its left adjoint $i_n$ also respects homotopy groups, i.e. $\pi_k(Y,a,b)\cong \pi_k(i_n Y,a,b)$ for every $k\leq n$.
\end{prop}
\begin{proof}
Since $(t_n X)_k=X_k$ for every $k\leq n-1$, the statement clearly holds for $k\leq n-2$. Now consider $\pi_{n-1}(X,a,b)$: this is the set $\{[H]\colon H\in X_{n-1}, H\colon a\rightarrow b\}$ where we are quotienting by the equivalence relation $H\sim H'$ if and only if there exists an $\alpha\colon H \rightarrow H'$. Clearly, this coincides with $\pi_{n-1}(t_n X,a,b)$, since quotienting $n$-cells of $X$ does not affect the equivalence relation.

Finally, since $n$-cells of $t_n X$ are $n$-cells of $X$ up to $(n+1)$-cells, we get:
\[\pi_n(X,a,b)\cong (t_n X (a,b))_0\cong \pi_n (t_n X a,b) \]
since, by definition, $t_n X$ has no non-identity $(n+1)$-cells, i.e. there exists $\gamma\colon H \rightarrow H'$ if and only if $H=H'$, when $H,H'$ are $n$-cells.

The proof for $i_n$ is entirely analogous and therefore it is left to the reader.
\end{proof}
Since $\pi_{n+1}(X,a,b)$ is always trivial for an $(n+1)$-coskeletal $\C$-model $X$, we have the following result.
\begin{cor}
The two classes of maps $\mathcal{W}$ and $\mathcal{W'}$ coincide.
\end{cor} 
The following result allows to transfer the canonical left semi-model structure on \(\ngpd{n}\) (if it exists, see Proposition \ref{model str}) to a Quillen equivalent one on \((n+1)\)-coskeletal \(\infty\)-groupoids.
%Now let's assume the canonical left semi-model structure on $\ngpd{n}$ whose generating cofibrations are given by $\{S^{k-1}\rightarrow D_k\}_{k\leq n}\cup \{(\mathbf{Id}_{D_n},\mathbf{Id}_{D_n})\colon \partial D_{n+1}\rightarrow D_n\}$ and whose generating trivial cofibrations are given by $\{\sigma_k\colon D_k \rightarrow D_{k+1}\}_{k\leq n-1}$.
\begin{thm} \label{thm:Quillen_equiv_cosk_truncated}
Assuming the canonical left semi-model structure on \(\ngpd{n}\) exists, there exists a cofibrantly generated left semi-model structure on the category $\mathbf{Mod}(\C)_{cosk_{(n+1)}}$ of $(n+1)$-coskeletal $\infty$-groupoids. Moreover, the adjunction in \eqref{tr-cosk adj on models} is a Quillen equivalence with respect to these semi-model structures.
\end{thm}
\begin{proof}
We define the set of generating cofibrations to be $I=\{S^{k-1}\rightarrow D_k\}_{k\leq n}$ and that of generating trivial cofibrations to be $J=\{\sigma_k\colon D_k\rightarrow D_{k+1}\}_{k\leq n}\cup\{\sigma\colon D_{n+1}\rightarrow S^{n+1}\}$. Here, we actually mean the image of these maps in $\mathbf{Mod}(\C)$ under the reflection functor $\iota_n^*$, so that the last one is the image of $\sigma_{n+1}\colon D_{n+1}\rightarrow D_{n+2}$. The class of weak equivalences is defined to be $\mathcal{W}$, and we observe that $t_n$ sends generating cofibrations (resp. trivial cofibrations) to cofibrations (resp. trivial cofibrations). Moreover, since $\mathcal{W}=\mathcal{W'}$, the only non trivial fact that we have to check is that $J-inj\cap \mathcal{W}=I-inj$, where $K-inj$ denote the class of arrows with the right lifting property with respect to a given class of arrows $K$. It is clear that a map in $\mathbf{Mod}(\C)_{cosk_{(n+1)}}$ is a fibration (resp. trivial fibration) if and only if is such as a map of $\infty$-groupoids, so we can use the argument for $\infty$-groupoids to show that a fibration is a weak equivalence if and only if it is a trivial fibration (see, for instance, Lemma 3.6 in \cite{EL2}).

Therefore, $t_n$ is a left Quillen functor, and thanks to Proposition \ref{pi_n compatibility}, we have that both $t_n$ and $i_n$ preserve and detect weak equivalences. This is enough to conclude that the adjunction is indeed a Quillen equivalence, since we have $t_n(\eta_X\colon X \rightarrow i_n \circ t_n X) \cong \mathbf{Id}_{t_n X}$, which implies that $\eta_X$ is a weak equivalence in $\mathbf{Mod}(\C)_{cosk_{(n+1)}}$. The counit happens to be an isomorphism for every $n$-groupoid $Y$ so this concludes the proof.
\end{proof}
\begin{rmk}
We observe that the existence of a left semi-model structure on $\ngpd{n}$ as described is not necessary to prove that the two homotopy theories modeled by, respectively, $(n+1)$-coskeletal $\infty$-groupoids and $n$-groupoids, are equivalent. Indeed, we still get an adjunction of relative categories of the form:
\begin{equation}
\begin{tikzcd}
(\mathbf{Mod}(\C)_{cosk_{(n+1)}},\mathcal{W}) \ar[r,bend left,"t_n"] \ar[r, phantom, "\perp" description, very near start]& (\ngpd{n},\mathcal{W}_n) \ar[l,bend left, "i_n"]
\end{tikzcd}
\end{equation}
%\xymatrixcolsep{1pc}
%\vcenter{\hbox{\xymatrix{
%			**[l](\mathbf{Mod}(\C)_{cosk_{(n+1)}},\mathcal{W}) \xtwocell[r]{}_{i_n}^{t_n}{'\perp}& **[r] (\ngpd{n},\mathcal{W}_n)}}}

with the property that both the unit and the counit are weak equivalences. This is enough to prove that this adjunction induces an equivalence of the $(\infty,1)$-categories associated with these two relative categories. In fact, this follows directly from Proposition 7.1.14 of \cite{Cis}.
\end{rmk}
%Observe that $n$-truncated models are characterized by the (necessarily unique) lifting property against the following set of maps:
%\[\{(\mathbf{Id}_{D_n},\mathbf{Id}_{D_n})\colon \partial D_{n+1}\rightarrow D_n, D_{n+1}\rightarrow D_n\plus{\partial D_{n+1}}D_{n+1}\}\] where the second map features in the following cocartesian square:
%\[
%\bfig
%\morphism(0,0)|a|/@{>}@<0pt>/<600,0>[\partial D_{n+1}`D_{n+1};j]
%\morphism(0,0)|a|/@{>}@<0pt>/<0,-400>[\partial D_{n+1}`D_{n};(\mathbf{Id}_{D_n},\mathbf{Id}_{D_n})]
%\morphism(600,0)|a|/@{>}@<0pt>/<0,-400>[D_{n+1}`D_n\plus{\partial D_{n+1}}D_{n+1};]
%\morphism(0,-400)|a|/@{>}@<0pt>/<600,0>[D_n`D_n\plus{\partial D_{n+1}}D_{n+1};]
%\efig 
%\]
%\end{rmk}

\section{Preliminaries on cylinder categories}

\label{sec:cylinder_cat}

In this section we review the theory of cylinder and pre-cylinder categories developed in \cite{Hen}.

\begin{defi} \label{def:pre-cylidner}
  A pre-cylinder category is a small\footnote{Though we will also occasionally consider some large pre-cylinder category.} category $\Ccal$ with a class of arrows called cofibrations and a class of arrows called weak equivalences, satisfying the following axioms:

  \begin{enumerate}
  \item Cofibrations and equivalences contain isomorphisms and are stable under compositions.
  \item Weak equivalences satisfies $2$-out-of-$6$, i.e. if $f g$ and  $g h$ are equivalences then $f,g,h$ and $f g h$ are equivalences as well.
  \item $\Ccal$ has an initial objects and every object is cofibrant.
  \item Pushout of cofibrations exists and are cofibrations.
  \item Given a diagram:

 \[   \begin{tikzcd}
      C \ar[d, "\sim"] & A \ar[hook,r] \ar[d, "\sim"]  \ar[l] & B \ar[d, "\sim"] \\
      C' & A' \ar[hook,r] \ar[l] & B'
 \end{tikzcd}
 \] 

where the vertical maps are weak equivalences and the hooked arrows are cofibrations, the induced map $C \coprod_A B \rightarrow C' \coprod_{A'} B'$ is a weak equivalence as well.
\end{enumerate}
\end{defi}

Morphisms of pre-cylinder categories are the functors that preserve cofibrations, weak equivalences, the initial object and pushout of cofibrations. Pre-cylinder categories form a $2$-category, that, in an appropriate $2$-categorical sense, has all limits and colimits (it is a presentable $2$-category). It also has a rich theory of ``freely constructed pre-cylinder categories'' (also to be interpreted in a $2$-categorical sense) for which we refer to section 3.2 of \cite{Hen}. Here is an example we will use frequently, a more general description of free object will be discussed in remark \ref{rk:Description_of_Free_pre_cyl_cat}:

\begin{constr} The free pre-cylinder category on an object $*$, denoted $F_*$, is the category of finite sets, where cofibrations are the monomorphisms and weak equivalences are the isomorphisms, and the object $*$ corresponds to the singleton. Indeed, a morphism $X: F_* \rightarrow \Ccal$ is entirely determined by the image of $*$, and  the image of a general finite set $S$ is given by:

\[ X(S) = \coprod_{s \in S } X(*) \]
\end{constr}

\begin{defi} A cylinder category is a pre-cylinder category in which:

\begin{enumerate}
\item Every trivial cofibration $j$ admits a retraction $r$, i.e. a map such that $r j = Id$.

\item Every object $X$ admits a cylinder object, i.e. a cofibration/weak equivalence factorization:

\[ X \coprod X \hookrightarrow IX \overset{\sim}{\rightarrow } X \]

of its co-diagonal map.

\end{enumerate}

\end{defi}

Note that this is essentially the same as a Brown category of cofibrant objects, up to two differences: the requirement that trivial cofibrations have a retraction, and the fact that we are requiring the $2$-out-of-$6$ condition instead of the weaker $2$-out-of-$3$ condition. Our cylinder categories are exactly the opposite of the path categories of \cite{VDB}.

\begin{ex}
If $M$ is a left semi-model category (or a Quillen model category) in which every object is fibrant then the category $Cof(M)$ of cofibrant objects of $M$, with its cofibrations and weak equivalences is a (large) cylinder category.
\end{ex}

\begin{rmk} \label{rk:equivalence_of_cylinders_cat} Morphisms of cylinder categories are the morphisms of the underlying pre-cylinder categories. Note that (as they are special case of Brown categories) there is a good notion of homotopy equivalence between two cylinder categories, i.e. morphisms that induce an equivalence of categories between the formal localizations of the two cylinder categories at their set of weak equivalences. Other characterizations of this notion of homotopy equivalences can be found in section 2.3 of \cite{Hen} (where they are called acyclic morphisms).
\end{rmk}

\begin{constr} \label{constr:Category_of_models}
Given a (small) pre-cylinder category $\Ccal$, one defines its completion, or category of models, $\Mod(\Ccal)$ to be the category of functors:

\[ \Ccal^{op} \rightarrow \mathbf{Set} \]
that send the initial object to the singleton and pushouts along cofibration to pullbacks of sets.

The Yoneda embedding $\Ccal \rightarrow \Mod(\Ccal)$ defines a functor that sends the initial object to the initial object and commutes with pushouts along cofibrations. We will always identify $\Ccal$ with its image under this functor.

We consider the weak factorization system on $\Mod(\Ccal)$ generated by the image (under the Yoneda embedding) of the cofibrations in $\Ccal$. Therefore, the right class (called trivial fibrations) consists of the maps with the right lifting property against all cofibrations in $\Ccal$, and the left class is equivalently described as the class of maps with the left lifting property against all trivial fibrations. These are the retracts of transfinite compositions of pushouts of cofibrations in $\Ccal$. We will also use the term ``cofibrations'' to refer to the maps in the left class.

We can now formulate the first main result in \cite{Hen} (see \cite{Ba} for the notion of left/right (semi) model structure):
\end{constr}

\begin{thm} \label{thm:model_str_on_category_of_models}
 If $\Ccal$ is a cylinder category, then its category of models admits a left semi-model structure where:

\begin{itemize}

\item The cofibrations/trivial fibrations weak factorization system is generated by the cofibrations in $\Ccal$ as in \ref{constr:Category_of_models}.

\item The trivial cofibrations/fibrations weak factorization system is generated by the trivial cofibrations in $\Ccal$.

\item Every object of $\Mod(\Ccal)$ is fibrant.

\end{itemize}

If $F: \Ccal \rightarrow \Dcal$ is a morphism of pre-cylinder categories between two cylinder categories, then it induces a Quillen adjunction:

\[ F_! : \Mod(\Ccal) \leftrightarrows \Mod(\Dcal) : F^* \]

Where $F^*$ is composition of functor with $F$ and $F_!$ is the left Kan extention of $F: \Ccal \rightarrow \Dcal \subset \Mod(\Dcal)$. If $F$ is a homotopy equivalence in the sense of \ref{rk:equivalence_of_cylinders_cat}, then this Quillen pair is a Quillen equivalence.

\end{thm}

We can formulate some sort of converse result to this theorem, asserting that if the category of models has a model structure satisfying few additional conditions, then the model structure is obtained from the construction above. This will be useful later.
\begin{lemma} \label{lem:cylinder_from_MS}
Let $\Ccal$ be a pre-cylinder category with a set of generating cofibrations $I$, in the sense that every cofibration is a finite composite of pushout of arrows in $I$, and whose class of weak equivalences is generated by a set $J$ of trivial cofibrations, in the sense that it is the smallest class of equivalences satisfying definition \ref{def:pre-cylidner} and containing $J$. Furthermore, assume that:

\begin{itemize}

\item $\Mod(\Ccal)$ carries a left semi-model structure in which the cofibrations are the usual cofibrations of the category of models, the generating trivial cofibrations are the arrows in $J$ and every object is fibrant.

\item Each generating cofibration $A \hookrightarrow B$ of $\Ccal$ admits a relative cylinder object $B \coprod_A B \hookrightarrow I_A B \overset{\sim}{\rightarrow} B $ which is also in $\Ccal$.

\end{itemize}

Then $\Ccal$ is a cylinder category.

\end{lemma}

\begin{proof}
The proof is essentially the same as proposition 5.3.6 of \cite{Hen}. The key observation is that given a cofibration $f:A \hookrightarrow B$ in $\Ccal$ which is obtained as an $n$-iterated pushout of the generating cofibrations, one can show by induction on $n$ (using the model structure on $\Mod(\Ccal)$) that it admits a relative cylinder object such that the map $B \hookrightarrow I_A B$ is an iterated pushout of the maps $D \hookrightarrow I_C D$ where $C \hookrightarrow D$ are the generating cofibrations appearing in the construction of $A \hookrightarrow B$. In particular every cofibration of  $\Ccal$ has a relative cylinder object in $\Ccal$. Every trivial cofibration in $\Ccal$ is automatically a trivial cofibration in $\Mod(\Ccal)$ as well, so as every object of $\Mod(\Ccal)$ is fibrant, it implies that every trivial cofibration of $\Ccal$ has a retraction, and this concludes the proof.
\end{proof}

The other main contributions of \cite{Hen} is the construction of a weak model structure (in the sense of \cite{Hen2}) on the category of pre-cylinder categories whose fibrant objects are the cylinder categories, and whose weak equivalences between fibrant objects are the homotopy equivalences of \ref{rk:equivalence_of_cylinders_cat}. This is closely related to the structure of category of fibrant objects on Brown categories of cofibrant objects constructed by Szumilo in \cite{Szu}.
\begin{rmk}
Note that we are committing a small abuse of language here: it is only a weak model structure in a $2$-categorical sense (i.e. all the liftings properties are only up to coherent isomorphism), moreover \cite{Hen2} was written long after \cite{Hen} so that the language of weak model categories was actually not explicitely used in \cite{Hen}, instead all the factorizations and lifting properties are stated explicitely. We refer the reader to section 4 of \cite{Hen} for further details.
\end{rmk}
More about this weak model structure is provided below. But first, let's introduce the main example of how this weak model structure can be used to construct and compare other model categories.
\begin{defi} \label{def:cylcoherator}
A cylinder coherator is a fibrant replacement of $F_*$, i.e. the free pre-cylinder category on one object, in the weak model structure on pre-cylinder category. This amounts to a trivial cofibration $F_* \overset{\sim}{\hookrightarrow} \Ccal$ with fibrant codomain.
\end{defi}

It follows from what we have explained so far that:

\begin{itemize}

\item The category of models of a cylinder coherator is endowed with a Quillen model structure. Indeed, cylinder coherators are fibrant pre-cylinder categories, hence cylinder categories.

\item Given two cylinder coherators, there is a canonical Quillen equivalence between their categories of models. Indeed they are both fibrant replacement of the same object, so there is a weak equivalence between them, and by the last claim of theorem \ref{thm:model_str_on_category_of_models} this induces a Quillen equivalence between their category of models.

\end{itemize}

Moreover, it is shown in \cite[section 5.2]{Hen} that a certain cylinder coherator has its category of models which is Quillen equivalent to the model category of spaces, so by the observation above they all share this property. This is the key ingredient of the proof in \cite{Hen} that if the canonical left semi-model structure on Grothendieck $\infty$-groupoids (for some coherator) exists (see Proposition \ref{model str}), then it must be Quillen equivalent to the model category of Spaces. It is relatively easy to see that Grothendieck $\infty$-groupoids (with their natural notion of cofibration) are the models of a certain pre-cylinder categories. Using lemma \ref{lem:cylinder_from_MS} one can then show that if the model structure exists, then this pre-cylinder category is a cylinder category. It is relatively easy to show (this argument will be reproduced in the following section) that if it is fibrant, then it is a cylinder coherator, and hence its ``model category of models'' is equivalent to the model category of spaces.

We now give some more precise description of this weak model structure on pre-cylinder categories:

To begin with, there are three generating cofibrations:

\begin{itemize}

\item The map $F_* \rightarrow F_{\hookrightarrow}$ that ``freely adds a cofibration with specified codomain'', i.e. $F_*$ is the free pre-cylinder on an object and $F_{\hookrightarrow}$  is the free pre-cylinder category on one cofibration and $F_* \rightarrow F_{\hookrightarrow}$ is the map sending the generating object to the codomain of the generating cofibration. $F_{\hookrightarrow}$ can be described explicitely as the category of finite set $X,Y$ with an arrow $X \rightarrow Y$, cofibrations are the square of monomorphisms and the generating cofibration is the unique commutative square whose top-left corner is  $\emptyset$ and the other three corners are singletons. Taking a pushout along this map freely adds a cofibration of specificed codomain to a given pre-cylinder category.

\item The map $F_{\hookrightarrow} \rightarrow F_r$ that freely adds a retraction to a given cofibration.

\item The map $F_{\rightarrow} \rightarrow F_{\sim}$ that freely turns an arrow into a weak equivalence.

\end{itemize}

\begin{rmk} \label{rk:Description_of_Free_pre_cyl_cat} The category of models $\Mod(\Ccal)$ of a pre-cylinder category $\Ccal$ can be described as the category of morphisms from $\Ccal$ to the pre-cylinder category $\mathbf{Set}^{op}$ in which every map is a cofibration and a weak equivalence. In particular, the category of models of the free pre-cylinder category on a pushout or of a freely constructed pre-cylinder category are easy to describe, using the universal property of such construction to describe what are morphisms into $\mathbf{Set}^{op}$. This is exploited in more detail in section 3.2 of \cite{Hen}. For example, if $\Ccal$ is a pre-cylinder category, $X \hookrightarrow Y$ is a cofibration in $\Ccal$ and one considers the pre-cylinder category $\Ccal^+$ freely obtained from $\Ccal$ by adding a retraction to $X \hookrightarrow Y$ (i.e. taking a pushout along the generating cofibration $F_{\hookrightarrow} \rightarrow F_r$ mentioned above), then a model of $\Ccal^+$ is the same as a model $M$ of $\Ccal$ such that the map $M(X) \rightarrow M(Y)$ is endowed with a retraction.

Continuing along these lines a little further, one can show that cofibrant pre-cylinder categories corresponds exactly to Cartmell's Generalized algebraic theory (see \cite{Cart}) in which there is no equality axiom but only type introduction axiom and term introduction axiom.

\end{rmk}

One has two generating anodyne morphisms of pre-cylinder categories:

\begin{itemize}

\item The map that freely adds a retraction to a given trivial cofibration.

\item The map that freely adds a factorization as a cofibration followed by a weak equivalence to a given morphisms.

\end{itemize}

Fibrant objects and fibration between fibrant objects are defined by the right lifting property against the generating anodyne morphisms. Note that the fibrant objects are clearly the cylinder categories. One then defines the trivial cofibrations as the cofibration with the left lifting properties against all fibrations between fibrant objects. Finally, the general fibrations are the morphisms with the right lifting property against all trivial cofibrations.

The fact that this forms a ``weak model category'' is summarized by the following properties: 

\begin{itemize}
\item Trivial fibrations are defined as the maps with the right lifting property against generating, hence all, cofibrations. Fibration are defined as having the right lifting propert against all trivial cofibrations. In particular one has all the expected lifting properties.

\item Equivalences are defined as the equivalences of homotopy categories between fibrant objects, and for arrows between objects that are either fibrant or cofibrant by using fibrant replacement of all the cofibrant objects.

\item For a map with fibrant target, being a trivial fibration is equivalent to being a fibration and an equivalence.

\item For a map with cofibrant domain being a trivial cofibration is equivalent to being a cofibration and an equivalence.

\item Any morphism with a fibrant target can be factored as a trivial cofibration followed by a fibration or as a cofibrations followed by a trivial fibration.
\end{itemize}

\begin{constr} \label{constr:C^eq} One of the key ingredients in the construction of this weak model structure is the existence of a ``path category'' for cylinder categories.
In fact, the main result of this paper will rely more on the existence of this path object than on the existence of the weak model structure.
 If $\Ccal$ is a cylinder category, one defines $C^{eq}$ to be the category whose objects are triple of object $X_1,X_2,X_I \in \Ccal$ endowed with a cofibration $X_1 \coprod X_2 \hookrightarrow X_I$ whose two component  $X_1 \hookrightarrow X_I$ and $X_2 \hookrightarrow X_I$ are trivial cofibrations. Morphisms are the natural transformation (i.e. triple of maps making the obvious diagram commute). Weak equivalences are the triple of weak equivalences. Cofibrations are the ``Reedy cofibrations'', i.e. the morphism $f:(X_1,X_2,X_I) \rightarrow (Y_1,Y_2,Y_I)$ such that $f_1:X_1 \rightarrow Y_1$ and $f_2:X_2 \rightarrow Y_2$ are cofibrations, and the latching map:

\[ X_I \coprod_{X_1 \coprod X_2} (Y_1 \coprod Y_2) \hookrightarrow Y_I \]

is a cofibration.

It is shown in section 4.2 and 4.3 of \cite{Hen} that $\Ccal^{eq}$ is also a cylinder category, that the map $\Ccal^{eq} \rightarrow \Ccal \times \Ccal $ forgeting the $X_I$ component is a fibration in the sense explained above, and the composites $\Ccal^{eq} \rightarrow \Ccal$ with the two projections are both trivial fibrations. Note that in general there is no morphism $\Ccal \rightarrow \Ccal^{eq}$ producing a factorization of the diagonal map of $\Ccal$, as it is in general not possible to choose $X_I$ functorially and in a way compatible with pushouts of cofibrations. Nevertheless, it is shown in section 4.3 of \cite{Hen} that such a map exists when $\Ccal$ is cofibrant. Hence we obtain a path objects for bifibrant objects, which is why one only gets a weak model structure.
\end{constr}

\begin{constr} \label{constr:Hslice} We will also need to use the ``homotopy slice'' construction for cylinder categories, which was developed in \cite[section 4.4]{Hen}. In general, the slice $\Ccal_{/X}$ is not a cylinder category, and even if it were, it might not represent the correct homotopy theory. In order to fix this, we replace them with the so-called homotopy slice $\Ccal^{X}$. This is the category of pairs of object $A,I \in \Ccal$ endowed with a cofibration $A \coprod X \hookrightarrow I$ such that the component $X \hookrightarrow I$ is a trivial cofibration. Morphisms are pairs of maps making the obvious diagram commute,  equivalences are pairs of equivalence and cofibrations are the Reedy cofibrations defined similarly to what was done in construction \ref{constr:C^eq}.

It is shown in section 4.4 of \cite{Hen} that if $\Ccal$ is a cylinder category then $\Ccal^{X}$ is also a cylinder category and $\Ccal^{X} \rightarrow \Ccal$ is a fibration. More generally, if $p:\Ccal \twoheadrightarrow \Dcal$ is a fibration and $X \in \Ccal$ then $\Ccal^X \rightarrow \Dcal^{P(X)}$ is also a fibration. Moreover $\Ccal^{X} \rightarrow \Ccal$ is a trivial fibration if and only if $X$ is h-terminal in the sense of \cite[Def. 4.4.3]{Hen}, i.e. if any solid diagram as below admits a filling as indicated by the dotted arrow:

\[\begin{tikzcd}
  A \ar[d,hook] \ar[r] & X \, , \\
B \ar[ur, dotted] 
\end{tikzcd}
\]
Equivalently, this happens if and only if $X$ is a terminal object in the homotopy category.

\end{constr}

We conclude this section with two examples of pre-cylinder categories that will be very important in the next section.

\begin{constr}

Let $rGlob_n$ denote the category of reflexive $n$-globular sets, i.e. $n$-globular sets endowed, for each $k$-cell $x$ with $k<n$, with the choice of a $k+1$-cell $r(x)$ with $x$ as source and target. The forgetful functor from reflexive $n$-globular sets to $n$-globular sets has a left adjoint that justs adds a cell $r^k(x)$ for each cell $x$ and positive integer $k$. One also denotes $\partial D_{k+1}$ and $D_k$ the image of the $\partial D_{k+1}$ and $D_n$ in the category of globular set as in Definition \ref{models defi} by this left adjoint functor. 

One takes the morphisms $\partial D_k \hookrightarrow D_k$ to be the generating cofibrations in the category $rGlob_n$, and it is easy to see that the cofibrations are exactly the monomorphisms in this case.

One let $rGlob^f_n$ be the pre-cylinder category whose underlying category is the category of finite reflexive $n$-globular sets, with cofibrations as above, and with  weak equivalences generated by the morphisms between the $D_k$ (i.e. the smallest class of equivalences containing these arrows and making this category into a pre-cylinder category).

Note that $rGlob^f_0$ is the same as $F_*$. One also has the follwing results.

\end{constr}

\begin{lemma} \label{lem:morphism_from_glob}
\begin{enumerate}
\item[]
\item For each $n$, the natural morphism:

\[ rGlob^f_n \rightarrow rGlob^f_{n+1} \]

is freely generated by adding a relative cylinder object to the cofibrations $\partial D_n \hookrightarrow D_n$. In particular it is an anodyne morphism.

\item The category of models of $rGlob^f_n$ is the category of reflexive $n$-globular sets, where the inclusion of $rGlob^f_n$ in reflexive $n$-globular sets is the Yoneda embedding.

\item A morphism $rGlob^f_n \rightarrow D$ into a pre-cylinder category $D$ is given by the choice of an object $X \in D$ and of the so-called first $n$ iterated relative cylinder objects of $X$, i.e. a sequence of objects \(I^0 X=X, \ \ldots, \ I^n X\) each fitting into a cofibration/weak equivalence factorization in $D$:

\[ I^{n-1} X \coprod_{I^{n-2} X \coprod I^{n-2} X } I^{n-1} X \hookrightarrow I^n X  \overset{\sim}{\rightarrow} I^{n-1} X \]

More precisely, $X$ is the image of $D_0$ and the factorization above is the image of the morphism:

\[\partial D_n=D_{n-1} \coprod_{\partial D_{n-1}} D_{n-1} \hookrightarrow D_n \rightarrow D_{n-1}. \]

\end{enumerate}

\end{lemma}

\begin{proof}
One proves all points by simultaneous induction, using the general description of free pre-cylinder categories mentioned in \ref{rk:Description_of_Free_pre_cyl_cat}. Assuming $2.$ and $3.$ hold up to dimensions $n$, we consider the pre-cylinder category $T$ obtained by freely adding a relative cylinder object for  $\partial D_n \hookrightarrow D_n$ to $rGlob^f_n$. In particular, a morphism from $T$ into any other pre-cylinder category $D$ is described as in point $3.$ (for $rGlob^f_{n+1}$). Using remark \ref{rk:Description_of_Free_pre_cyl_cat} that models of $T$ can be described as morphisms from $T$ to the (opposite) of the category of sets, one obtains that the models of $T$ are reflexive $(n+1)$-globular sets. Indeed, a model of $T$ is given by a model $X$ of $rGlob^f_n$, i.e. a reflexive $n$-globular set, together with a set $X_{n+1}$ fitting into a factorization of the form:

\[ X(D_n) \rightarrow X_{n+1} \rightarrow X(D_n ) \times_{X(D_{n-1}) \times X(D_{n-1})} X(D_n) \]

This is is the same as a reflexive $(n+1)$-globular set. Finally, $T$ identifies with the models of $T$ that are finitely generated by pushout of the generating cofibrations, and those are by definition exactly the finite $(n+1)$-globular sets, with equivalences freely generated by those of $rGlob^f_n$ and the fact that morphisms between $D_n$ and $D_{n+1}$ are equivalences (so that $D_{n+1}$ is indeed a relative cylinder object), so it is indeed isomorphic to $rGlob^f_{n+1}$.

\end{proof}

\begin{constr} \label{constr:CDelta} A construction very similar to the one done above for globular sets, applied to categories of presheaves over a directed category, is developed in section 4.2 of \cite{Hen}. We will explicitly need the special case of semi-simplicial sets that we briefly present here.

Let $\Delta_+$ be the category of finite non-empty ordinals and injective order preserving maps between them. Presheaves on $\Delta_+$ are called semi-simplicial sets. We denote by $\Ccal^{\Delta_+}_0$ the pre-cylinder category with the following universal property: for any other pre-cylinder category $D$, morphisms from  $\Ccal^{\Delta_+}_0$ to $D$ are given by the data of a Reedy cofibrant diagram $\Delta_+ \rightarrow D$ such that all maps are sent to equivalences.

Alternatively, it can be described as the category of finite semi-simplicial sets where the cofibrations are the monomorphisms and the equivalences are the smallest set of arrows containing the morphisms between representable objects and making $\Ccal^{\Delta_+}_0$ into a pre-cylinder category.

It is shown in \cite[Prop. 5.2.11]{Hen} that the morphisms $F_* \rightarrow \Ccal^{\Delta_+}_0$ corresponding to $\Delta_+[0]$ is a trivial cofibration.

Note, in \cite{Hen} we have shown that the semi-simplicial horn inclusion $\Lambda^k_n[n] \hookrightarrow \Delta_+[n]$ are trivial cofibrations in $\Ccal^{\Delta_+}_0$, in particular if one freely add a retraction to all of them, one obtains a new pre-cylinder category denoted $\Ccal^{\Delta_+}$ together with an anodyne morphisms $\Ccal^{\Delta_+}_0 \rightarrow \Ccal^{\Delta_+}$. The models of $\Ccal^{\Delta_+}$ are the semi-simplicial algebraic Kan complexes (i.e. semi-simplicial endowed with chosen lift against all horn inclusion), and it is proved in \cite{Hen} proposition 5.2.9 that $\Ccal^{\Delta_+}$ is a cylinder category, and hence as the morphism $F_* \rightarrow \Ccal^{\Delta_+}$ is anodyne, it is a cylinder coherator.

\end{constr}

\section{Main theorem}
\label{sec:mainTh}

The central point of the argument in \cite{Hen} in order to show that if the canonical left semi-model structure on $\infty$-groupoids exists (see Proposition \ref{model str}) then it is equivalent to the model category of spaces is to show that both these model categories have the homotopy universal property to be `` (homotopicaly) freely generated by one object''. This is formalized by the notion of cylinder coherator, i.e. of fibrant replacement of the pre-cylinder category $F_*$ freely generated by one object as in Definition \ref{def:cylcoherator}.

In order to adapt this argument to $n$-groupoids and homotopy $n$-types one needs to find a similar universal property for these categories. Our candidate for this property is that they are ``freely generated by a single $n$-co-truncated object'' where the notion of $n$-co-truncated object is an object $X$ satisfying the equivalent condition of the proposition below:

\begin{prop} Let $M$ be a model category in which every object is fibrant, and let $X$ be a cofibrant object in $M$. The following conditions are equivalent:

\begin{enumerate}

\item Given a semi-co-simplicial cofibrant resolution of $X$, i.e. a morphism $F:\Ccal_0^{\Delta_+} \rightarrow M$ sending $\Delta_+[0]$ to $X$, the  cofibration $F(\partial \Delta_+[n+2] ) \hookrightarrow F(\Delta_+[n+2])$ is an equivalence.

\item Given a choice of the first $n+1$ iterated relative cylinder objects of $X$ as in lemma \ref{lem:morphism_from_glob}, i.e. a morphism $I:rGlob^f_{n+1} \rightarrow X$ sending $D_0$ to $X$, the morphism $I(\partial D_{n+2}) \rightarrow I(D_{n+1})$ (or equivalently any of the two cofibration $I(D_{n+1}) \hookrightarrow I(\partial D_{n+2})$) is an equivalence.

\end{enumerate}

Moreover, in both cases, the conditions do not depend on the choice of the morphisms $I$ and $F$.

\end{prop}

As said above, objects satisfying these equivalent conditions will be called $n$-co-truncated objects.

One can actually see that these two conditions are both equivalent to the fact that $X$ is a $n$-co-truncated\footnote{by co-truncated, we mean truncated in the opposite category.} object in the sense of \cite{HTT} definition 5.5.6.1 in the $\infty$-category associated with $M$, and that this even holds without assuming that every object in $M$ is fibrant. Indeed, for the equivalence with the first point it essentially follows from lemma 5.5.6.17 of \cite{HTT} together with the fact that $F(\partial \Delta[n+2])$ is a model of $\partial \Delta[n+2] \otimes X$, and for the second point it follows from lemma 5.5.6.15 of \cite{HTT} together with the fact that iterated cylinders are model for the iterated codiagonal\footnote{The codiagonal map of $f:X \rightarrow Y$ is the map $Y \coprod_X Y \rightarrow Y$, the $n$-th iterated codiagonal map of a map $f$ is the codiagonal map of its $(n-1)$-th iterated codiagonal maps. The iterated codiagonal of an object $X$ is the iterated codiagonal of the map $\emptyset \rightarrow X$ where $\emptyset$ is the initial object.  In a model category these are modeled by iterated cylinders.} maps.

In order to make the paper more self contained, we give a different proof relying instead on the machinery of cylinder coherators.
\begin{proof}

Both $F_* \overset{\sim}{\hookrightarrow} \Ccal_0^{\Delta_+}$ and $F_* \overset{\sim}{\hookrightarrow} rGlob^f_{n+1}$ are trivial cofibrations of pre-cylinder categories, which allows for any given object $X \in M$ to construct functors $F$ and $I$ as in the proposition. We will show that given any two such functors, conditions $(1)$ for $F$ is equivalence to conditions $(2)$ for $I$, which in particular shows that the conditions $(1)$ and $(2)$ do not depend on the choice of $F$ and $I$.  The morphism:

\[ F_* \hookrightarrow  \Ccal_0^{\Delta_+} \coprod_{F_*} rGlob^f_{n+1} \]

where the coproduct is taken in the category of pre-cylinder categories, is again a trivial cofibration. One can then construct a fibrant replacement:
\[ F_* \overset{\sim}{\hookrightarrow} \Ccal_0^{\Delta_+} \coprod_{F_*} rGlob^f_{n+1} \overset{\sim}{\hookrightarrow} T \]

Any two choices of functors $F$ and $I$ as in the proposition give rise to a morphism

 \[ \Ccal_0^{\Delta_+} \coprod_{F_*} rGlob^f_{n+1} \rightarrow Cof(M) \]

 which can be extended to a morphism $T \rightarrow Cof(M)$.

Now $T$ is a cylinder coherator, for it is a fibrant replacement of $F_*$, hence its category of models is Quillen equivalent to the category of spaces, with the equivalence given by any morphisms $T \rightarrow Cof(Spaces)$ sending $\Delta[0]$ (which in $T$ is isomorphic to $D_0$) to the point. In particular this functor can be chosen to send $\Delta_+[k]$ to the standard simplex and $\partial D_k$ and $D_k$ to the standard spheres and balls.

Condition $(1)$ can be rephrased as the fact that $F(\Delta[n+1]) \rightarrow F(\partial \Delta[n+2])$ is an equivalence, and in the category of spaces the maps: $\Delta[n+1] \rightarrow \partial \Delta[n+2]$, and $D_{n+1} \rightarrow \partial D_{n+1}$ are homeomorphic, hence they are homotopic in $T$ as well. In particular, their images in $M$ are homotopic, and this proves the equivalence of conditions $(1)$ and $(2)$.

\end{proof}

\begin{constr} \label{cstr:groupoid_pre-cyl}
Let $\C$ be a cellular globular theory, with defining tower:

\[ \Theta_0 = \C_1 \rightarrow \C_2 \rightarrow \dots \rightarrow \C_n \rightarrow \dots \rightarrow \C \]

We will construct a sequence of pre-cylinder categories:

\[ F_* \hookrightarrow \C^{0}_a \hookrightarrow \dots \hookrightarrow \C^{n}_a \hookrightarrow \dots \hookrightarrow \C_a \]

such that the $\C^{n}_a$-models are the same as the $n$-coskeletal models of $\C$ % I do want \C here \C_a is not defined yet
( or equivalently of $\C_n$), with its usual notion of cofibration, and the morphisms above are cofibrations of pre-cylinder categories (and corresponds to the usual adjunctions induced by the restriction morphisms). The $\C^n_a$ have no equivalences except for isomorphisms.

Explicitely, $\C_a^n$ is the category of finitely generated (in a polygraphic sense, i.e. by pushout of the generating cofibrations) $n$-coskeletal $\C$-groupoids. One still call $\partial D_{k+1}$ and $D_k$ the object of $\C^n_a$ freely generated by the globular sets with the same name.

Assuming one has already proved that $\C_a^{n}$ has all the desired property, one defines $\C_a^{n+1}$ from $\C_a^n$ as follows: first, one freely add the object $D_{n+1}$ with a cofibration $\partial D_{n+1} \hookrightarrow D_{n+1}$.  Then for each of the new $(n+1)$-dimensional operation $f_i$ appearing in $\C_{n+1 } = \C_{n}[f_i]$, one freely adds an arrow as the dotted arrow fitting in the diagram:

\[\begin{tikzcd}
 \partial D_{n+1} \ar[d,hook] \ar[rr,"(f_i\circ \sigma{,}f_i\circ \tau)"] && P \\  
D_{n+1} \ar[urr,dotted,"f_i"{swap}] 
\end{tikzcd} \]

where $P$ is the globular sum corresponding to $f_i$.

Assuming by induction that the models of $\C^{n}_a$ are indeed the $n$-coskeletal models of $\C_{n}$, models of  $\C^{n+1}_a$ (which is defined above as a free object) have an additional object of $(n+1)$-arrows with source and target map as expected, and all the new generators of $\C_{n+1}$ that acts on them. By freeness of $\C_{n+1}$ this makes the models of $\C^{n+1}_a$ exactly the same as $(n+1)$-coskeletal models of $\C_{n+1}$.

By its free construction, the usual sphere inclusions are generating cofibrations of $\C^{n+1}_a$, and hence the object of $\C^{n+1}_a$ are indeed the finitely freely generated $(n+1)$-coskeletal $\C_{n+1}$-models.

One now defines $\C^{n}_a \hookrightarrow \C^n_h$ by forcing some morphisms to be equivalence (so far, $\C^n_a$ has no equivalences other than the isomorphisms). More precisely the new equivalences are the morphisms:

\[ s: D_k \rightarrow D_{k+1} \quad k<n \]
\[ i_0 \colon D_n \rightarrow \partial D_{n+1}\]

Note that the morphisms $\C^n_a \hookrightarrow \C^{n+1}_a$ are not morphisms from $\C^n_h$ to $\C^{n+1}_h$, but the morphism $F_* \rightarrow \C^{n+1}_h$ corresponding to $D_0$ is a cofibration.

\end{constr}

In the rest of the article we will always assume that $\C$ is contractible, i.e. it is a coherator for $\infty$-groupoids. In pratice we only ever use $\C_{n+1}$, so it is enough to assume contractibility up to dimension $n+1$.

\begin{prop} \label{prop:CanMS<=>Cylinder}
The following are equivalent:

\begin{enumerate}

\item $\C_h^{n+1}$ is a cylinder category.

\item The canonical left semi-model structure on $(n+1)$-coskeletal $\Ccal$-models exists (see Proposition \ref{model str}).
\end{enumerate}

\end{prop}

\begin{proof} $(1)$ implies $(2)$ follows from the construction of the left semi-model structure on the models of a cylinder category, as recalled in theorem \ref{thm:model_str_on_category_of_models}. The converse follows from lemma \ref{lem:cylinder_from_MS} together with the fact that each generating cofibration $\partial D_{k+1} \rightarrow D_{k+1}$ of $\C_h^{n+1}$ has relative cylinder object given by:

 \[ D_{k+1} \coprod_{\partial D_{k+1}} D_{k+1} = \partial D_{k+2} \hookrightarrow D_{k+2} \overset{u}{\rightarrow} D_{k+1} \qquad k<n \]
\[ D_{n+1} \coprod_{\partial D_{n+1}} D_{n+1} = \partial D_{n+2} \hookrightarrow \partial D_{n+2} \rightarrow D_{n+1} \]

where $u$ is any ``unit'' morphism (i.e. a common retraction of \(\sigma,\tau\colon D_{k+1} \rightarrow D_{k+2}\)) provided by the contractibility of $\C$.
\end{proof}

\begin{prop} \label{prop:lift_hcontractible}  Let $p \colon A \rightarrow B$ be a fibration between cylinder categories. Consider a lifting problem of the form:
\[
\begin{tikzcd}
 F_*  \ar[r,"x"] \ar[d,hook] & A \ar[d,"p"] \\
 \C^{n+1}_h \ar[r] & B
\end{tikzcd}
\]
such that the object $x$ is $n$-co-truncated. We also assume that
\begin{center} $(*)$ the objects $x$ and $p(x)$ are both h-terminal, as defined in \ref{constr:Hslice} (or  \cite[definition $4.4.3$]{Hen}). \end{center}
Then the square admits a diagonal lift.
\end{prop}

\begin{proof}
The proof is exactly the same as lemma 5.3.9 of \cite{Hen}, i.e. one first lifts the action of globular sets, then one lifts one by one the generating operations of $\C^{n+1}_h$. The assumption that $x$ and $p(x)$ are h-terminal allows to construct all the required operations. See also the proof of lemma 5.3.9 of \cite{Hen} for more details.

The condition that $x$ is $n$-co-truncated is used to ensure that the morphism that we constructed this way do send the last generating equivalence $D_{n+1} \rightarrow \partial D_{n+2}$ to an equivalence, i.e. it is indeed a morphism of pre-cylinder categories $\C^{n+1}_h \rightarrow A$, and not just $\C^{n+1}_a \rightarrow A$.
\end{proof}
\begin{prop} \label{prop:lift_globular_conditionel} Assuming that the canonical left semi-model structure on $n$-truncated $\Ccal$-models exists, then Proposition \ref{prop:lift_hcontractible} holds even without assuming $(*)$.
\end{prop}

\begin{proof} The proof is the same as the proof of Corollary 5.3.12 of \cite{Hen}.
Assuming the existence of the canonical left semi-model structure we have shown ( in Proposition \ref{prop:CanMS<=>Cylinder}) that $\C^{n+1}_h$ is a cylinder category. In particular, one can apply the construction of the ``homotopy slice'' developed in section 4.4 of \cite{Hen} to it. Moreover, by contractibility of $\C$, the object $D_0$ in $\C^{n+1}_h$ is equivalent to the terminal object, hence it is h-terminal. In particular, the morphism:

\[ (\C^{n+1}_h)^{D_0} \rightarrow \C^{n+1}_h \]

is a trivial fibration, so it admits a section sending the object $D_0$ to the object $(D_0,D_1) \in (\C^{n+1}_h)^{D_0}$.

Now, given a square of the form:

\[
\begin{tikzcd}
 F_*  \ar[r,"x"] \ar[d,hook] & A \ar[d,"p"] \\
 \Ccal^{n+1}_h \ar[r] & B
\end{tikzcd}
\]

such that the object corresponding to $x$ is $(n-1)$ co-truncated, using functoriality of the homotopy slice construction and any choice of lift of $x$ one obtains a new diagram:

\[
\begin{tikzcd}
 F_*  \ar[rr,"x'"] \ar[d,hook] & & A^x \ar[r] \ar[d,"p'"] & A \ar[d,"p"] \\
 \C^{n+1}_h \ar[r,"s"]  & (\C^{n+1}_h)^{D_0}  \ar[r] & B^{p(x)} \ar[r] & B
\end{tikzcd}
\]

where the map $s$ is the section mentioned above, and $x'$ is any appropriately chosen lift of the cylinder object $(D_0,D_1)$ into a cylinder object for $x$, which exists since fibrations lift cylinder objects. The map $p'$ is a fibration (see \ref{constr:Hslice}, or \cite[proposition 4.4.2]{Hen}), hence one can apply Proposition \ref{prop:lift_hcontractible} to get a lift in the first square, which produces a lift of our initial square and concludes the proof.
\end{proof}
Proposition \ref{prop:lift_hcontractible}, and its strengthening in Proposition \ref{prop:lift_globular_conditionel}, produce, assuming that the canonical left semi-model structure for $n$-groupoids exists, a (homotopical) universal property for the cylinder category of $\C^{n+1}_h$. Namely, it is (homotopically) freely generated by a $n$-co-truncated object $D_0$. Indeed, given an $n$-co-truncated object $B$ in a cylinder category $X$, the lifting property as in \ref{prop:lift_hcontractible} and \ref{prop:lift_globular_conditionel} applied to $X \rightarrow 1$ gives a morphism $\C^{n+1}_h \rightarrow X$ sendind $D_0$ to $B$. Now given two such morphisms  one can apply the lifting property to the fibration $X^{eq} \rightarrow X \times X$ where $X^{eq}$ is the path object for cylinder categories constructed in 4.2.8 of \cite{Hen} to get a homotopy between these two maps. Similarly, one gets higher homotopies between homotopies by looking at lifting against iterated path objects. 

The general idea of our main result is that any other cylinder category with the same property will be equivalent. So we just need to show that any other cylinder category with the same universal property is equivalent to the category of $n$-truncated spaces to conclude the proof. One cannot directly use a Bousfield localization of the category of spaces or of simplicial sets because they are not ``cofibrant''. So, as in \cite{Hen}, we use a semi-simplicial model instead.
\begin{constr}
We start from the pre-cylinder category $\Ccal^{\Delta_+}_0$. While it is not a cylinder category, its category of models, the category of semi-simplicial sets, carries a right semi-model structure whose cofibrations are the monomorphisms (i.e. the natural cofibration of the category of models) and the fibrations  between fibrant objects are characterized by the lifting property against the trivial cofibration of $\Ccal^{\Delta_+}_0$ (see for example \cite{Hen2}).
Moreover, the forgetful functor from simplicial sets to semi-simplicial sets is both a left and a right Quillen equivalence. Along with the fact that any fibrant semi-simplicial set can be endowed with the structure of a simplicial sets (which, in particular, is a Kan complex), this allows to express almost all the homotopy theory of semi-simplicial sets in terms of simplicial sets.

This right semi-model structure admits\footnote{This can be checked directly, but it require some work with semi-simplicial sets that is outside the scope of this paper. A more detailed treatement of Bousfield localizations of weak and right semi-model categies easily implying this claim will appear in an upcoming paper.} a Bousfield localization at the morphism $\partial \Delta_+[n+2] \hookrightarrow \Delta_+[n+2]$, whose local objects are the (homotopically) $n$-truncated semi-simplicial sets. This localisation is Quillen equivalent to the localization of the category of simplicial sets at $\partial \Delta[n+2] \hookrightarrow \Delta[n+2]$, which is known to be a model for the homotopy theory of $n$-truncated spaces, i.e. homotopy \(n\)-types.

Finally, one defines the pre-cylinder category $\Ccal^{\Delta_+}_{\leqslant n}$  by starting from $\Ccal^{\Delta_+}_0$, then forcing $\partial \Delta_+[n+2] \hookrightarrow \Delta_+[n+2]$ to be an equivalence, and finally one adds a retract to each trivial cofibration. In fact, it is enough to add a retract to each of the generating trivial cofibrations of the localized right semi-model structure mentioned above.

The category of models of $\Ccal^{\Delta_+}_{\leqslant n}$ is the category of algebraically fibrant objects of the right semi-model structure on semi-simplicial sets modeling $n$-truncated spaces. In particular one can show, applying the argument of \cite{Nik}, that it carries a model structure transported from the model structure on semi-simplicial sets (a treatment of the model structure on fibrant objects on a weak and right semi-model categories will appear in a paper in preparation by John Bourke and the first named author). Lemma \ref{lem:cylinder_from_MS} allows to show that $\Ccal^{\Delta_+}_{\leqslant n}$ is a cylinder category.

It is also clear that $\Ccal^{\Delta_+}_{\leqslant n}$ satisfies the lifting property of proposition \ref{prop:lift_hcontractible} (without assuming the condition $(*)$). Indeed, given such lifting problem, one can first lift $\Ccal^{\Delta_+}_0$, then the assumption that the object above is $n$-cotruncated implies that the lift will send $\partial \Delta[n+2] \hookrightarrow \Delta[n+2]$ to an equivalence, and then the lifting property of fibrations allows to construct the lift of the retracts of trivial cofibrations.
\end{constr}
\begin{thm}
	\label{main thm}
Assuming the canonical left semi-model structure on $(n+1)$-coskeletal groupoid exists, then any left Quillen functor from it to the model category of spaces localized at $\partial D_{n+2} \rightarrow D_{n+2}$ that sends $D_0$ to the point is a Quillen equivalence. Moreover such Quillen functor exists.
\end{thm}
\begin{proof}
One first shows that the cylinder categories $\Ccal^{\Delta_+}_{\leqslant n}$ and $\C^{n+1}_h$ are equivalent.

Their respective lifting properties imply that we have morphisms in both directions:

\[ \Ccal^{\Delta_+}_{\leqslant n} \leftrightarrows \Ccal^{n+1}_h \]

under which $D_0$ is sent to (the free object on) $\Delta_+[0]$ and viceversa. This is the case because both $D_0$ and $\Delta_+[0]$ are $n$-co-truncated. In particular, the composite gives a morphism  $f: \C^{ n+1}_h \rightarrow  \C^{n+1}_h$ sending $D_0$ to $D_0$, so that one can construct, using Proposition \ref{prop:lift_hcontractible} a diagonal lift for the square:

\[
\begin{tikzcd}
 F_* \ar[r,"D_1"] \ar[d] &   (\C^{n+1}_h)^{eq} \ar[d] \\
 \C^{n+1}_h \ar[r,"(f;Id)" swap] \ar[ur,dotted,"h"] &  \C^{n+1}_h \times  \C^{n+1}_h
\end{tikzcd}
\]

Here,  $(\Ccal^{\leqslant n+1}_f)^{eq}$ is the path object presented in \ref{constr:C^eq}, and the top arrow (denoted by $D_1$) corresponds to the object $(D_0,D_0,D_1) \in (\C^{n+1}_h)^{eq} $ with the natural cofibration $D_0 \coprod D_0 \hookrightarrow D_1$. A lift in this square shows that $f$ is homotopic to the identity (for instance, its action on the homotopy category is isomorphic to the identity). Since the same is true for $\Ccal^{\Delta_+}_{\leqslant n}$, this shows that any two morphisms as above are equivalences, and so they induce Quillen equivalences between the categories of models.
In particular, any co-semi-simplicial resolution of $D_0$ in $\C^{n+1}_h$, corresponding to a morphism $\Ccal^{\Delta_+}_0 \rightarrow \C^{n+1}_h$, induces an equivalence between the model category of coskeletal $(n+1)$-groupoids constructed in \ref{thm:Quillen_equiv_cosk_truncated} and the localization of the right semi-model structure of semi-simplicial sets at $\partial \Delta[n+2] \hookrightarrow \Delta[n+2]$.

One can explicitely construct a Quillen functor from  $\Ccal^{n+1}_h$ to the model category of spaces localized at $\partial D_{n+1} \hookrightarrow D_{n+1}$ by sending $D_0$ to $D_0$. In this case, the right adjoint corresponds to the foundamental $n$-groupoid construction. It does not quite follows from the lifting property of proposition \ref{prop:lift_hcontractible} as not every object is fibrant in the localized category of spaces, but the construction in the proof of proposition \ref{prop:lift_hcontractible} still produces a morphisms $\C^{n+1}_a \rightarrow Spaces$ and one can justs check that once the category of spaces is localized at $\partial D_{n+1} \hookrightarrow D_{n+1}$ this morphism is compatible with equivalences.

Any composite:
\[ \Ccal^{\Delta_+}_{0} \rightarrow \C^{n+1}_h \rightarrow Spaces \]
will send the representables to a co-semi-simplicial resolution of the point. All such resolutions are equivalent and induce equivalences between the category of semi-simplicial sets and the category of spaces, and all specialize to equivalences between the localization of semi-simplicial sets at $\partial \Delta[n+2] \hookrightarrow \Delta[n+2]$ and the localization of spaces at $\partial D_{n+2} \rightarrow D_{n+2}$. By the $2$-out-of-$3$ property for Quillen equivalences, this concludes the proof.
\end{proof}

\end{document}